\documentclass[11pt,a4paper,twoside]{article}
\usepackage{amsmath}
\usepackage{amsthm}
\usepackage{amssymb}
\usepackage{fancyhdr,caption,multicol}
\usepackage{pstricks,pst-eucl,pst-grad,pst-poly,pstricks-add}
\usepackage{pstricks}
\usepackage{pst-func}
\usepackage{pst-grad}
\usepackage{psfig}
\usepackage{color}
\usepackage{pst-plot}
\usepackage{multido}
\usepackage{graphicx}
\usepackage{hyperref}
\usepackage{epic,eepic}

\newtheorem{ley}{Ley}[section]
\newtheorem{thm}[ley]{Theorem}
\newtheorem{cor}[ley]{Corollary}
\newtheorem{lem}[ley]{Lemma}
\newtheorem{prop}[ley]{Proposition}

\newtheorem{rem}[ley]{Remark}
\newtheorem{ex}[ley]{Example}

\newcommand{\R}{\mathbb{R}}
\newcommand{\N}{\mathbb{N}}

\title{\bf Dines theorem for inhomogeneous quadratic functions and nonconvex
optimization}
\title{\bf Pair of inhomogeneous quadratic functions: joint-range convexity
properties and nonconvex programming}
\title{\bf Joint-range convexity for a pair of inhomogeneous quadratic functions
and applications to QP}

\author{
Fabi\'an Flores-Baz\'an \thanks{Departamento de Ingenier\'{i}a Matem\'atica,
CI$^2$MA, Universidad de Concepci\'on, Casilla 160-C,
Concepci\'on, Chile ({\tt fflores@ing-mat.udec.cl}). The research
for the first author was supported in part by CONICYT-Chile through
FONDECYT and BASAL Projects, CMM, Universidad de Chile.}
\and Felipe Opazo \thanks{Departamento de Ingenier\'{i}a Matem\'atica,
Universidad de Concepci\'on, Casilla 160-C,
Concepci\'on, Chile ({\tt felipeopazo@udec.cl})}
}

\date{}

\begin{document}

\maketitle
\begin{abstract}
We establish various extensions of the convexity Dines theorem for a
(joint-range) pair of inhomogeneous quadratic functions. If convexity fails
we describe those rays for which the sum of the joint-range and the ray is convex.
These results are suitable for dealing nonconvex
inhomogeneous quadratic optimization problems under one quadratic equality 
constraint. As applications of our main results, different sufficient conditions
for the validity of S-lemma (a nonstrict version of Finsler's theorem) for 
inhomogenoeus quadratic functions, is presented.
In addition, a new characterization of strong duality under Slater-type condition is 
established.
\end{abstract}

\begin{quote} {\small {\bf Key words.} Dines theorem, Nonconvex optimization, hidden
convexity, Quadratic programming, S-lemma, nonstrict version of Finsler's theorem,
Strong duality.}
\end{quote}
\begin{quote} {\small {\bf Mathematics subject classification 2000.}
Primary: 90C20, 90C46, 49N10, 49N15, 52A10.}
\end{quote}

\section{Introduction}\label{secc1}

Quadratic functions has proved to be very important in mathematics because of its
consequences in various subjects like calculus of variations,
mathematical programming, matrix theory (related to matrix pencil),
geometry and special relativity \cite{heste-mc1940, ham1999, toep1918, haus1919,
greub1967, hj1985, barvinok1995, elton2010}, among others, and applications in
Applied sciences: telecommunications, robust control \cite{mslt2008, sdl2006}, trust
region problems \cite{gay, sor}.

The lack of convexity always offers a nice challenge in mathematics,
but sometimes, as occurs in the quadratic world, hidden convexity is present,
It seems to be that one of the first results for quadratic forms is due to Finsler
\cite{fins1937}, known as (strict) Finsler's theorem, which refers to positive
definiteness of a matrix pencil. The same result was proved, independently, by
the Chicago's School under the guidance of Bliss. We quote Albert
\cite{albert1938}, Reid, \cite{reid1938}, Dines \cite{dines1941}, Calabi
\cite{calabi1964}, Hestenes \cite{heste1968}.

It perhaps the first beautiful results  for a pair of quadratic
forms is due to Dines \cite{dines1941} and Brickman \cite{brickman1961},
proving the convexity, respectively, of
\begin{equation}\label{din00}
\{(\langle Ax,x\rangle,\langle Bx,x\rangle)\in\R^2:~x\in\R^n\},
\end{equation}
\begin{equation}\label{bric00}
\{(\langle Ax,x\rangle,\langle Bx,x\rangle)\in\R^2:~\langle x,x\rangle=1,~
x\in\R^n\}~~(n\geq 3),
\end{equation}
provided $A$ and $B$ are real symmetric matrices. Actually Dines, motivated by the
above result due to Finsler, searched the convexity in \eqref{din00}. This
convexity property inspired to many researchers for searching hidden convexity
in the quadratic framework. Generalizations to more than two matrices were
developed in \cite{barvinok1995, pol1998, ham1999, hb-torki2002,
derin2006, polter2007}, and references therein, without being completed.
It is well known that, in general, $(f,g)(\R^n)$ is nonconvex if $f$ and $g$ are
inhomogeneous quadratic functions. \\
Precisely, our interest in the present paper is to consider a pair of
inhomogeneous quadratic functions $f$ and $g$, and to describe completely when the
convexity of $(f,g)(\R^n)$ occurs (the only result we aware is
Theorem 2.2 in \cite{pol1998}, it will be contained in our Theorem
\ref{lema:connd} below). In addition, we also answer
the question about which directions $d$ we
must add to the set $(f,g)(\R^n)$ in order to get convexity, in another words, for
which directions $d$, the set $(f,g)(\R^n)+\R_+d$ is convex.
As a consequence of our main result we
recover the Dines theorem. We exploit the hidden convexity to derive some sufficient
condition for the validity of an S-lemma with an equality constraint (a nonstrict 
version of Finsler's theorem for inhomogeneous quadratic functions), which are
expressed in a different way than that established in \cite{xws2015}, suitable for 
dealing with the problem
\begin{equation}\label{prob-min100}
\inf\{f(x):~g(x)= 0,~ x\in \R^n\}.
\end{equation}
The latter S-lemma is also useful for dealing with bounded generalized trust region 
subproblems,  that is, with constraints $l\leq g(x)\leq u$, as shown in 
\cite{xws2015}.
   
Moreover, a new strong duality result for this problem as well as
necessary and sufficient optimality conditions are established, covering situations
where no result in \cite{more, jleeli2009, xws2015} is applicable.
In \cite{xws2015}, by using a completely different approach, a characterization of
the convexity of $(f,g)(\R^n)$, when $g$ is affine, is given.

A complete description (besides the convexity) of the set
${\rm cone}((f,g)(\R^n)-\mu(1,0)+\R_+^2)$, where
\begin{equation}\label{prob-min200}
\mu\doteq \inf\{f(x):~g(x)\leq 0,~ x\in \R^n\},
\end{equation}
for any pair of inhomogeneous quadratic functions $f$ and $g$, is given in
\cite{ffb-carcamo2013} by assuming $\mu$ to be finite; and when $\mu=-\infty$ the
set ${\rm cone}((f,g)(\R^n)+\R_+^2)$ is considered. When $f$ and $g$ are any
real-valued functions, strong duality for \eqref{prob-min200} implies the
convexity of $\overline{\rm cone}((f,g)(\R^n)-\mu(1,0)+\R_+^2)$ as shown in
\cite{fb2v}.

It is worthwhile mentioning that the  existence of solution for
\eqref{prob-min200} was fully analyzed in \cite{bental2014} under simultaneous
diagonalizability (SD).

We point out that the convexity of $C\doteq (f,g)(\R^n)+\R_+^2$ (proved in Theorem
\ref{dines:00} below) was stated in Corollary 10 of  \cite{tt2013}, but its proof
is not correct since the set $C$ is not closed in general:
Examples \ref{ej:reff} and  \ref{ej:s:lema}  show this fact. On the other hand, we 
mention the recent paper \cite{jl2014} where it is proved, under suitable 
assumptions, the convexity of $(f,g_0,h_1,\ldots,h_m)(\R^n)+\R_+^{m+2}$ with $f$ 
being any quadratic function, $g_0$ (quadratic) strictly convex and all the other 
functions $h_i$ affine linear.
Another joint-range convexity result involving $Z$-matrices may be found in
\cite{jleeli2009}.

Apart from the characterizations of strong duality, several sufficient conditions
of the zero duality gap for convex programs have been established in the
literature, see \cite{fk, a2000, at2003, z, bw2006, bcm2008, bcw2008, tseng2009,
ozda2006}.

The paper is structured as follows. Section \ref{secc2}  provides the necessary
notations, definitions and some preliminaries to be used throughout
the paper: in particular, the Dines theorem is recalled. Some characterizations
of bi-dimensional Simultaneous Diagonalization (SD) and Non Degenerate (ND)
properties for a pair of matrices are established in Section \ref{secc3}.
Section \ref{secc4} contains our main results, all of them related to extensions of
Dines theorem. Applications of those extensions to nonconvex quadratic optimization
under a single equality constraint are presented in Section \ref{secc5}: they
include a new S-lemma (a nonstrict version of Finsler's theorem for inhomogeneous
quadratic functions), strong duality results, as well as necessary and sufficient
optimality conditions. Finally, Section \ref{secc6} presents, for reader's
convenience, a brief historical note about the appearance, in a chronological
order, of the several properties arising in the study of quadratic forms. Some
relationships between those properties are also outlined.

\section{Basic notations and some preliminaries}\label{secc2}

In this section we introduce the basic definitions, notations and some
preliminary results.

Given any nonempty set $K\subseteq\R^n$, its closure
is denoted by $\overline{K}$; its convex hull  by ${\rm co}(K)$ which is the
smallest convex set containing $K$; its topological interior by ${\rm int}~K$,
whereas its relative interior by ${\rm ri}~K$, it is the interior with respect to
its affine set; the (topological) boundary of $K$ is denoted by ${\rm bd}~K$. We
denote the complement of $K$ by ${\mathcal C}(K)$.
We set ${\rm cone}(K)\doteq\displaystyle{\bigcup_{t\geq 0}}tK$, being the smallest
cone containing $K$, and
$\overline{{\rm cone}}(K)\doteq\overline{\displaystyle{\bigcup_{t\geq 0}}tK}.$
In case $K=\{u\}$, we denote ${\rm cone}~K=\R_+u$ and 
$\R u\doteq\{tu:t\in\R\}$, where $\R_+\doteq[0,+\infty[$.
Furthermore, $K^*$ stands for the (non-negative) polar cone of $K$ which is
defined by
$$K^*\doteq\{\xi\in\R^n:~\langle \xi,a\rangle\geq 0~~\forall~a\in K\},$$
where $\langle\cdot,\cdot\rangle$ means the scalar or inner product in $\R^n$,
whose elements are considered column vectors. Thus, $\langle a,b\rangle=a^\top b$
for all $a, b\in\R^n$. By $K^\perp$ we mean the ortogonal subspace to $K$, given by
$K^\perp=\{u\in\R^n:~\langle u,v\rangle=0~~\forall~v\in K\}$; in case
$K=\{u\}$, we simply put $u^\perp$; $\R_+u$ stands for the ray starting at the
origin along the direction $u$. We say $P$ is a cone if $tP\subseteq P$ for all
$t\geq 0$, and it is pointed if $P\cap(-P)=\{0\}$.


Throughout this paper the matrices are always with real entries.
Given any matrix $A$ or order $m\times n$, $A^\top$ stands for the transpose of
$A$; whereas if $A$ is a symmetric square matrix of order $n$, we say it is 
positive semidefinite, denoted by $A\succcurlyeq 0$, if $\langle Ax,x\rangle\geq 0$ 
for all $x\in\R^n$; it is positive definite, denoted by $A\succ 0$ if
$\langle Ax, x\rangle>0$ for all $x\in\R^n$, $x\not=0$. The set of symmetric
square matrices of order $n$ is denoted by ${\mathcal S}^n$. \\
Given any quadratic function
$$f(x)= \langle Ax,x\rangle+\langle a,x\rangle+k_1,$$
for some $A\in{\mathcal S}^n$, $a\in\R^n$ and $k_1\in\R$, we set
$$f_H(x)\doteq\langle Ax,x\rangle,~~f_L(x)\doteq \langle a,x\rangle.$$
If we are given another quadratic function
$$g(x)=\langle Bx,x\rangle+\langle b,x\rangle+k_2,$$
for some $B\in{\mathcal S}^n$, $b\in\R^n$ and $k_2\in\R$. Set
\begin{equation}
z_{u,v}\doteq \begin{pmatrix}
\langle Au, v\rangle\cr
\langle Bu, v\rangle
\end{pmatrix},~~F_H(u)\doteq \begin{pmatrix}
f_H(u)\cr
g_H(u)
\end{pmatrix}.
\end{equation}

An important property in matrix analyis and in the study of nonconvex
quadratic programming, is that of Simultaneous Diagonalization property.
We say that any  two matrices $A, B$ in ${\mathcal S}^n$ has the Simultaneous
Diagonalization (SD) property, simple simultaneous diagonalizable, if there
exists a nonsingular matrix $C$  such that both $C^\top AC$ and $C^\top BC$
are diagonal
\cite[Section  7.6]{hj1985}, that is, if there are linearly independent (LI)
vector $u_i\in\R^n$, $i=1,\ldots,n$, such that
$z_{u_i,u_j}=0, ~~i\not=j.$ Such an
assumption, for instance, allowed the authors in \cite{bental1996} to
re-write the original problem in a more tractable one. The symbol LD stands for 
linear dependence.

It is said that $A$ and $B$ are Non Degenerate (ND) if
\begin{equation}\label{nd:00}
\langle Au,u\rangle=0=\langle Bu,u\rangle\Longrightarrow u=0.
\end{equation}

One of the most important results concerning quadratic functions refers to
Dine's theorem \cite{dines1941}, it perhaps motivated by Finsler's theorem
\cite{fins1937}.

\begin{thm}\label{din1941}
$\cite[Theorem~1, Theorem~2]{dines1941}\cite[Theorem~2]{heste1968}$
The set $F_H(\R^n)$ is a convex cone. In addition, if \eqref{nd:00} holds then
either $F_H(\R^n)=\R^2$ or $F_H(\R^n)$ is closed and pointed.
\end{thm}

The convexity may fail for $F(\R^n)$ if $F(x)=(f(x),g(x))$ with $f, g$ being
not necessarily homogeneous quadratic functions, as the next example shows.
\begin{ex} \label{ex0}
Consider $f(x_1,x_2)=x_1+x_2-x_1^2-x_2^2-2x_1x_2$,
$g(x_1,x_2)=x_1^2+x_2^2+2x_1x_2-1$, and define
the set $M\doteq\{(f(x_1,x_2),g(x_1,x_2))\in\R^2:~(x_1,x_2)\in\R^2\}$.
Clearly $(0,0)=(f(0,1),g(0,1))\in M$ and
$(-2,0)=(f(-1,0),g(-1,0))\in M$, but
$(-1,0)=\frac{1}{2}(0,0)+\frac{1}{2}(-2,0)~\notin~F(\R^2)$.
One can actually see that
$$F_H(\R^2)=\R_+(-1,1),~{\rm and}~F(\R^2)=\{(t-t^2,t^2-1):~t\in\R\}.$$
\end{ex}
Another instance is Example \ref{ej:op0}, where
$$F(\R^2)=\{(0,0\}\cup [\R^2\setminus(\R\times\{0\})],~~F_H(\R^2)=
\R\times\{0\}.$$


We now state a simple result which will be used in the next
sections. For any $u=(u_1,u_2)\in\R^2$, set $u_\perp\doteq (-u_2,u_1)$, so that
$\|u\|=\|u_\perp\|$ and $\langle u_{\bot}, u\rangle =0$.

\begin{prop} \label{felo0} Let $u, v\in\R^{2}$. The following hold
\begin{itemize}
\item[$(a)$]  $\langle v_{\bot}, u\rangle =-\langle u_{\bot}, v\rangle;$
\item[$(b)$]  $\langle v_{\bot}, u\rangle \neq0\Longleftrightarrow\{u, v\} $
is LI.
\item[$(c)$] Assume that $\{ u, v\} $ is LI. Then
\begin{enumerate}
\item[$(c1)$] $h=t_{1}u+t_{2}v,~t_{2}\geq 0~({\rm resp.}~t_{2}>0)~
\Longleftrightarrow~\langle u_{\bot}, v\rangle \langle u_{\bot}, h\rangle \geq0~
({\rm resp.}~>0);$
\item[$(c2)$] $h=t_{1}u+t_{2}v, t_{1}\geq 0, t_2\geq 0\Longleftrightarrow
\langle u_{\bot}, v\rangle \langle u_{\bot}, h\rangle \geq0~{\rm and}~
\langle v_{\bot}, u\rangle \langle v_{\bot}, h\rangle \geq0.$
\end{enumerate}
\end{itemize}
\end{prop}

Finally, the next lemma which is important by itself  will play an important role
in the subsequent sections.
\begin{lem}\label{denso} Let $X\subseteq\R^n$ be a nonempty subset of $\R^n$ 
and $h\not=0$, $h_1$, be any elements in $\R^2$ such that
\begin{equation}\label{ine:denso}
F(X)+\R h+\R_+ h_1\subseteq F(\R^n).
\end{equation}
Then $F(\R^n)$ is convex under any of the following circumstances:
\begin{itemize}
\item[$(a)$] $\{ h_{1}, h\}$ is LI and $\overline{X}=\R^n;$
\item[$(b)$] $\{ h_{1}, h\}$ is LD and $X=\R^n$.
\end{itemize}
\end{lem}
\begin{proof} Let $0<t<1$ and $x, y\in \R^n$ with $F(x)\neq F(y)$.
The desired result is obtained by showing that
$f_t\doteq tF(x)+(1-t)F(y)\in F(\R^n)$.\\
$(a)$: By assumption $\langle h_\bot,h_1\rangle\not=0$, and therefore,
from \eqref{ine:denso} and Proposition \ref{felo0} one gets, for all $x_0\in X$,
\begin{equation}\label{eq:01}
H(x_0)\doteq \{ u:~\langle h_\bot,h_1\rangle\langle h_\bot, u-F(x_{0})\rangle>0\}
\subseteq F(\R^n).
\end{equation}
The desired result is obtained by showing that
$f_t\doteq tF(x)+(1-t)F(y)\in H(x_0)$ for some $x_0\in X$. We distinguish two
cases.\\
$(a1)$: $\langle h_\bot,h_1\rangle\langle h_\bot, F(y)-F(x)\rangle >0$ (the case
``$<$'' is entirely similar). Since
$$
\langle h_\bot,h_1\rangle\langle h_\bot, f_t-F(x)\rangle>0,
$$
by densedness and continuity, we get $\bar x\in X$ close to $x$ such that
$f_t\in H(\bar x)$, and so $f_t\in F(\R^n)$ by \eqref{eq:01}.\\
$(a2)$: $\langle h_\bot,h_1\rangle\langle h_\bot, F(y)-F(x)\rangle =0$.
Let us consider the functions $q_1:\R\to\R^2$ and $q:\R\to\R$ defined by
$$
q_1(\lambda)\doteq F(\lambda x+(1-\lambda)y),~~
q(\lambda)\doteq \langle h_\bot,h_1\rangle\langle h_\bot, q_1(\lambda)-F(x)\rangle.
$$
Clearly $q$ is quadratic satisfying $q(0)=q(1)=0$.  Let us consider first that
$q\equiv 0$. Due to continuity $q_{1}([0, 1])$ is a connected set contained
in the line $F(x)+\R h$ passing through $F(x)$ and $F(y)$. Thus,
$f_t \in q_{1}([0, 1])\subseteq F(\R^n)$. \\
We now consider $q\not\equiv 0$. Then there exists  $\lambda_{1}\in\R$
satisfying $q(\lambda_{1})<0$, i. e.,
$$\langle h_{\bot},\, h_{1}\rangle \langle h_{\bot},\, F(\lambda_{1}x+
(1-\lambda_{1})y)-f_{t}\rangle =
\langle h_{\bot}, h_{1}\rangle \langle h_{\bot}, F(\lambda_{1}x+
(1-\lambda_{1})y)-F(x)\rangle<0.$$
Hence by taking $\bar x\in X$ near $\lambda_{1}x+(1-\lambda_{1})y$, we obtain
$\langle h_{\bot}, h_{1}\rangle \langle h_{\bot}, f_{t}-F(\bar x)\rangle >0$, and
so $f_{t}\in F(\R^n)$ by \eqref{eq:01}.\\
$(b)$: As $\{ h_{1}, h\} $ is LD, then \eqref{ine:denso} means that for all
$x_0\in Y$,
$$H_0(x_0)\doteq\{u\in\R^2:~\langle h_{\bot}, u-F(x_0)\rangle =0\}\subseteq
F(\R^{n}).$$
Let  $q(\lambda)=\langle h_{\bot}, F(\lambda x+(1-\lambda)y)-f_{t}\rangle $.
Then $q$ is continuous and $q(0)=t\langle h_{\bot}, F(y)-F(x)\rangle $,
$q(1)=(1-t)\langle h_{\bot}, F(x)-F(y)\rangle $. We observe that either
$q(0)=0=q(1)$ or $q(0)q(1)<0$. In the first case $q(\lambda)=0$ for all
$\lambda\in\R$, and so $f_t\in F(\R^n)$. In case of opposite sign, we get
$\lambda_0\in~]0,1[$ such that $q(\lambda_0)=0$, which implies that
$f_t\in H_0(\lambda_0 x+(1-\lambda_0)y)\subseteq F(\R^n)$.
\end{proof}

\section{Characterizing SD and ND in two dimensional spaces}\label{secc3}

This section is devoted to characterizing the simultaneous diagonalization and
non degenerate properties for a pair of matrices in terms of its homogeneous
quadratic forms. As one may found in the literature, the study in $\R^2$ deserves
a special treatment from $\R^n$, $n\geq 3$, and to the best knowledge of these
authors the following characterizations are new. As said before, here
$A,~B\in{\mathcal S}^2$.

We start by a simple proposition appearing elsewhere whose proof is presented here
just for reader's convenience.

\begin{prop} \label{sev:impl} Let us consider the assertions:
\begin{itemize}
\item[$(a)$] $F_H(\R^2)=\R^2;$
\item[$(b)$] {\rm ND} holds for $A$ and $B;$
\item[$(c)$] $F_H(\R^2)$ is closed.
\end{itemize}
Then $(a)\Longrightarrow(b)\Longrightarrow(c)$.
\end{prop}
\begin{proof} $(a)\Rightarrow(b)$: Let $u\in\R^2$ satisfying $F_H(u)=0$.
If on the contrary $u\not=0$, then by taking $v\in\R^2$ such that
$\{u, v\}$ is linearly independent, we obtain for $\alpha,~\beta\in\R$,
$$F_H(\alpha u+\beta v)=\alpha^2 F_H(u)+\beta^2 F_H(v)+2\alpha\beta z_{u,v}.$$
Thus $F_H(\R^2)\subseteq \R_+ F_H(v)+\R z_{u,v}$, which is impossible if
$F_H(\R^2)=\R^2$.\\
$(b)\Rightarrow(c)$: Let $F_H(x_k)$ be a sequence such that $F_H (x_k)\to z$.
In case $\|x_k\|$ is bounded, there is nothing to do. If $\|x_k\|$ is unbounded,
up to a subsequence, we may suppose that $\|x_k\|\to+\infty$ and
$\dfrac{x_k}{\|x_k\|}\to u$. Thus $\|u\|=1$ and
$$\dfrac{1}{\|x_k\|^2}F_H(x_k)=F_H(\dfrac{x_k}{\|x_k\|})\to F_H(u)=0,$$
which yields, by assumption, $u=0$, a contradiction.
\end{proof}

Example \ref{ej:op00} below shows that $(a)\Longrightarrow(b)$ may fail in higher
dimension. However, for $n\geq 3$, one obtains that $(a)$ implies the existence
of $u\in\R^n$, $u\not=0$, such  that $F_H(u)=0$, as Corollary in
\cite[p. 401]{heste1968} shows. We also point out the proof for proving $(b)$
implies $(c)$ remains valid for any dimension, see also Theorem 6 in
$\cite{heste1968}$.

\begin{ex}\label{ej:op00} Take
$$A=\begin{pmatrix}
1 & 0 & 0\cr
0 & 0 & 0\cr
0 & 0 & -1
\end{pmatrix},~~~
B=\begin{pmatrix}
0 & 0 & 0\cr
0 & 1 & 0\cr
0 & 0 & -1
\end{pmatrix},$$
Then, $F_H(\R^3)=\R^2$, but ND does not hold for $A$ and $B$.
\end{ex}

Next result provides a new characterization for SD in two dimension.

\begin{thm} \label{thm:0}
The following statements are equivalent:
\begin{itemize}
\item[$(a)$] {\rm SD} holds for $A$ and $B;$
\item[$(b)$] $\exists$ $u, v\in\R^2$ such that $F_H(\R^2)=\R_+u+\R_+v;$
\item[$(c)$] $F_H(\R^2)$ is closed and $F_H(\R^2)\not=\R^2$.
\end{itemize}
\end{thm}
\begin{proof} $(a)\Longrightarrow(b)$: By assumption, there exist LI vectors
$x, y\in\R^2$, such that $z_{x,y}=0$. Thus,
$F_H(\R^2)=\{F_H(\alpha x+\beta y):~\alpha, \beta\in\R\}$.
From the equality  $F_H(\alpha x+\beta y)=\alpha^2 F_H(x)+\beta^2 F_H(y)$
the desired conclusion is obtained.\\
$(b)\Longrightarrow(c)$: it is straightforward.\\
$(c)\Longrightarrow(a)$: We already know that $F_H(\R^2)$ is a convex cone.
We first check that $F_H(\R^2)$ cannot be a halfspace. Indeed, suppose that
$F_H(\R^2)=\{y\in\R^2:~\langle p, y\rangle\geq 0\}$ for some $p\in\R^2$,
$p\not=0$. Then there exist $u, v\in\R^2$ such that $F_H(u)=p_\perp$ and
$F_H(v)=-p_\perp$, which imply that $\{u, v\}$ is LI. Since
$F_H(\alpha u+\beta v)=\alpha^2 F_H(u)+\beta^2F_H(v)+2\alpha\beta z_{u,v}$  for all
$\alpha, \beta\in\R$, we get $2\alpha\beta\langle p,z_{u,v}\rangle\geq 0$ for all
$\alpha, \beta\in\R$. Hence $\langle p, z_{u,v}\rangle=0$, and therefore
$F_H(\R^2)=\{y\in\R^2:~\langle p, y\rangle= 0\}$.\\
Thus, the set $F_H(\R^2)$ may be $(i)$ the origin $\{0\}$; $(ii)$ a ray; $(iii)$ a 
pointed cone,  $(iv)$ a straightline.  \\
$(i)$: We simply take any two LI vectors
$u$ and $v$. Indeed, since $F_H(u)=F_H(v)=F_H(u+v)=0$, we obtain $z_{u,v}=0$.\\
$(ii)$: Assume that $F_H(\R^2)=\R_+p$, and take $u\in\R^2$ such that
$F_H(u)=p$, and choose $v\in\R^2$ so that $\{u,v\}$ is LI. In case
$z_{u,v}\not=0$, we proceed as follows. Since
$F_H(u+v)=F_H(u)+F_H(v)+2z_{u,v}$, we obtain
$0= \langle p_{\perp}, z_{u,v}\rangle$, which implies that
$z_{u,v}=\lambda p$
for some $\lambda\in\R$. It follows that $z_{u,v-\lambda u}=0$ with
$\{u,v-\lambda u\}$ being LI, and therefore SD holds.\\
$(iii)$: We have, for some LI vectors $p, q$ (see Proposition \ref{felo0})
\begin{equation}\label{rep:f:h}
F_H(\R^2)=\R_+p+\R_+q=\{z\in\R^2:
\langle p_\perp,q\rangle\langle p_\perp,z\rangle\geq 0,~
\langle q_\perp,p\rangle\langle q_\perp,z\rangle\geq 0\},
\end{equation}
with the property $\langle p_\perp,q\rangle=-\langle q_\perp,p\rangle\not=0$.
Take $u, v$ in $\R^2$ satisfying $F_H(u)=p$, $F_H(v)=q$. It follows that
$u$ and $v$ are LI. From \eqref{rep:f:h}, we get
in particular,
$\langle p_\perp,q\rangle\langle p_\perp,F_H(t u+v)\rangle\geq 0$,
for all $t\in\R$. This implies that
$\langle p_\perp,z_{u,v}\rangle= 0$. Similarly one obtains
$\langle q_\perp,z_{u,v}\rangle= 0$. Thus $z_{u,v}=0$, which is the desired
result.\\
$(iv)$: This case is similar to $(ii)$. Take $u, v\in\R^2$ such that
$F_H(u)=p_\perp$, $F_H(v)=-p_\perp$, which imply that $\{u, v\}$ is LI. Hence
$\{u,v-\lambda u\}$ is LI for some $\lambda\in\R$ and $z_{u,v-\lambda u}=0$.
\end{proof}

Next example illustrates that $u$ and $v$ need  not to be LI in the
previous theorem; Example \ref{ej:op00} shows that $(a)$ does not imply
$(b)$ in higher dimension, since we get 
$\R^2=F_H(\R^3)=\R_+(1,0)+\R_+(0,1)+\R_+(-1,-1)$, and
clearly SD holds for $A$ and $B$; whereas Example \ref{ej:reff} exhibits 
an instance where without the closedness of $F_H(\R^2)$ the implication
$(c)$ $\Longrightarrow$ $(a)$ may fail.  
\begin{ex} \label{ej:op1} Take
$$A=\begin{pmatrix}
0 & 1\cr
1 & 0
\end{pmatrix},~~~
B=0,
$$
Then, by choosing
$$C=\begin{pmatrix}
1 & 1\cr
-1 & 1
\end{pmatrix},
$$
we get that $C^\top AC$ is diagonal. It is easy to see that
$$F_H(\R^2)=\R_+(1,0)+\R_+(-1,0)=\R\times\{0\}.$$
\end{ex}

\begin{ex}\label{ej:reff} Consider
$$A=\begin{pmatrix}
1 & 1\cr
1 & 1
\end{pmatrix},~
B=\begin{pmatrix}
1 & 0\cr
0 & -1
\end{pmatrix}.
$$
Then, even if 
$$F_H(x_1,x_2)=(x_1+x_2)^2(1,0)+(x_1^2-x_2^2)(0,1),$$
one obtains $F_H(\R^2)=(\R_{++}\times \R)\cup\{(0,0)\}$, which is not closed and
clearly SD does not  hold for $A$ and $B$.
\end{ex}

We are now in a position to establish a new characterization for ND in $\R^2$.
\begin{thm}\label{nd0:eq}
The following assertions are equivalent:
\begin{itemize}
\item[$(a)$] {\rm ND} holds for $A$ and $B;$
\item[$(b)$] ${\rm ker}~A\cap{\rm ker}~B=\{0\}$ and $F_H(\R^2)$ is a closed set
different from a line.
\end{itemize}
\end{thm}
\begin{proof} $(a)\Longrightarrow(b)$: The first part of $(b)$ is straightforward,
and the closedness of $F_H(\R^2)$ is a consequence of Proposition \ref{sev:impl}.
It remains only to prove that $F_H(\R^2)$ is different from a line. In case
$F_H(\R^2)=\R^2$, we are done; thus suppose that $F_H(\R^2)\not=\R^2$. By
Theorem \ref{thm:0}, we have SD, that is, there exist $u, v\in\R^2$, LI, such that
$z_{u,v}=0$. This means $F_H(\alpha u+\beta v)=\alpha^2 F_H(u)+\beta^2 F_H(v)$ for
all $\alpha, \beta\in\R$. Obviously $F_H(u)\not=0\not=F_H(v)$, and if
$F_H(u)=-\lambda^2 F_H(v)$ for some $\lambda\not=0$, then $F_H(u+\lambda v)=0$.
This implies that $u+\lambda  v=0$ which is impossible, therefore
$F_H(u)\not=-\lambda^2 F_H(v)$ for all $\lambda\not=0$. Thus
$F_H(\R^2)=\{\alpha^2 F_H(u)+\beta^2 F_H(v):~\alpha, \beta \in\R\}$
is not a line.\\
$(b)\Longrightarrow(a)$: Since $F_H(\R^2)$ is closed, by Theorem \ref{thm:0},
either $F_H(\R^2)=\R^2$ or SD holds. In the first case, Proposition
\ref{sev:impl} implies that $(a)$ is satisfied. Assume that SD holds, as before,
there exist $u, v\in\R^2$, LI, such that $z_{u,v}=0$. Let $w\in\R^2$ satisfying
$F_H(w)=0$, we claim that $w=0$. By writting $w=\lambda_1 u+\lambda_2 v$ for
some $\lambda_i\in\R$, $i=1, 2$, we get
$F_H(w)=\lambda_1^2F_H(u)+\lambda_2^2F_H(v)$. We distinguish various situations.
If $F_H(u)=0$ (resp. $F_H(v)=0$), then $\langle Au,u\rangle=0$ and
$\langle Bu,u\rangle=0$
(resp. $\langle Av,v\rangle$ and $\langle Bv,v\rangle=0$), which
along with $\langle Au,v\rangle=0$ and $\langle Bu,v\rangle=0$, allow us to infer
$Au=0=Bu$ (resp. $Av=0=Bv$). It follows that $u=0$ (resp. $v=0$), which is
impossible.\\
We now consider $F_H(u)\not=0\not=F_H(v)$. Suppose, on the contrary, that
$\lambda_i\not=0$ for $i=1, 2$. Then, from
$F_H(w)=\lambda_1^2F_H(u)+\lambda_2^2F_H(v)=0$, we obtain
$F_H(u)=-\lambda F_H(v)$ for some $\lambda>0$. This yields that
$F_H(\R^2)=\{\alpha^2 F_H(u)+\beta^2 F_H(v):~\alpha, \beta \in\R\}$
is a line, a contradiction. Hence $\lambda_i=0$ for $i=1, 2$, and so $w=0$,
completing the proof.
\end{proof}

The same proof of the previous theorem allows us to obtain the next result
which establishes a relationship between  ND and SD.
\begin{cor}\label{cor0:eq}
The following assertions are equivalent:
\begin{itemize}
\item[$(a)$] $F_H(\R^2)\not=\R^2$ and {\rm ND} holds$;$
\item[$(b)$] ${\rm ker}~A\cap{\rm ker}~B=\{0\}$, $F_H(\R^2)$ is different from
a line and {\rm SD} holds;
\item[$(c)$] $\exists$ $(\lambda_1,\lambda_2)\in\R^2$ such that
$\lambda_1 A+\lambda_2 B\succ 0$.
\end{itemize}
\end{cor}
\begin{proof} $(a)\Longrightarrow(b)$ follows from Theorems \ref{thm:0} and
\ref{nd0:eq}; whereas the reverse implication is derived from the proof of
the previous theorem. The equivalence between $(a)$ and $(c)$ is Corollary 1 in
\cite[page 498]{dines1941}, valid for all $n\geq 2$.
\end{proof}

\section{Dines-type theorem for inhomogeneous quadratic functions and
relatives}\label{secc4}

This section is devoted to proving a generalization of Dines theorem for
inhomogeneous quadratic functions. Set
\begin{equation}\label{bi:quad0}
f(x)\doteq \langle Ax,x\rangle+\langle a,x\rangle,~
g(x)\doteq \langle Bx,x\rangle+\langle b, x\rangle,
\end{equation}
and, as before $F(x)=(f(x),g(x))$, so that $F(0)=(0,0)$.

We first deal with the one-dimensional case and afterward the general situation.
\subsection{The case of one-dimension}
We begin with the following useful simple result.
\begin{prop}\label{co:00} Let $u\in \R^n$, $u\not=0$. Then
\begin{itemize}
\item[$(a)$] $F(\R u)=\{\alpha^2F_H(u)+\alpha F_L(u):~\alpha\in\R\};$
\item[$(b)$] ${\rm co}~F(\R u)=F(\R u)+\R_{+}F_{H}(u)$.
\end{itemize}
\end{prop}
\begin{proof} $(a)$ is straightforward and $(b)$ is a consequence of the following
equalities:
\begin{eqnarray}\label{eco:01}
tF(\alpha u)&+&(1-t)F(\beta u)=[t\alpha^{2}+(1-t)\beta^{2}]F_H(u)+
[t\alpha+(1-t)\beta]F_L(u)\notag \\
&=&[(t\alpha+(1-t)\beta)^{2}+(t-t^2)(\alpha-\beta)^2]F_H(u)+
[t\alpha+(1-t)\beta]F_L(u)\notag \\
&=& F((t\alpha+(1-t)\beta)u)+(t-t^2)(\alpha-\beta)^2F_H(u).
\end{eqnarray}
\end{proof}

The one-dimensional version of (inhomogeneous) Dines-type theorem is expressed in
the following

\begin{lem}\label{lema01}
Let $u\in \R^n$, $u\not=0$ and $0\not=d\in\R^2$. The following hold:
\begin{itemize}
\item[$(a)$] Assume that $\{ F_{H}(u), F_{L}(u)\}$ is LD then $F(\R u)$ is convex.
\item[$(b)$] Assume that $\{ F_{H}(u), F_{L}(u)\}$ is LI. Then
\begin{enumerate}
\item[$(b1)$] if $d=F_{H}(u)$ one has $F(\R u)+\R_{+}d=
{\rm co}~F(\R u)+\R_{+}d={\rm co}~F(\R u);$
\item[$(b2)$] if $d=-F_{H}(u)$ then
$$F(\R u)+\R_{+}d=
F(\R u)\cup{\mathcal C}({\rm co}~F(\R u))=
\overline{{\mathcal C}({\rm co}~F(\R u))}\neq {\rm co}~F(\R u)+\R_{+}d;$$
\item[$(b3)$] if $\{ d, F_{H}(u)\}$ is LI, one has
$F(\R u)+\R_{+}d={\rm co}~F(\R u)+\R_{+}d\neq {\rm co}~F(\R u).$
\end{enumerate}
\end{itemize}
Similar results hold for the set $F(x+\R u)+\R_{+}d$ for any fixed
$x\in\R^n$ since
$$
F(x+tu)=t^2F_H(u)+t\left(\begin{array}{c}\langle 2Ax+a,u\rangle\\
\langle 2Bx+b,u\rangle
\end{array}\right)+F(x).
$$
\end{lem}
\begin{proof} We write $F(tu)=t^2F_H(u)+tF_L(u)$.\\
$(a)$: In this case the set $F(\R u)$ is either a point or ray or a line, so
convex.\\
$(b1)$: From Proposition \ref{co:00}, we obtain
\begin{equation}\label{eco:200}
{\rm co}(F(\R u)+\R_{+}d)={\rm co}~F(\R u)+\R_{+}d=
F(\R u)+\R_{+}F_H(u)+\R_+d,
\end{equation}
from which the convexity of $F(\R u)+\R_{+}d$ follows if $d=F_H(u)$.\\
$(b2)$: We obtain the following equalities, thanks to the LI of
$\{F_{H}(u), F_{L}(u)\}$:
\begin{eqnarray}
{\rm co}~F(\R u)&=&\Big\{ \sum_{i=1}^3\lambda_{i}F(\alpha_{i}u):~
\sum_{i=1}^3\lambda_{i}=1,~\lambda_{i}\geq0,~\alpha_i\in\R\Big\}\notag \\
&=&\Big\{\sum_{i=1}^3\lambda_{i}\alpha_{i}^{2}F_{H}(u)+
\sum_{i=1}^3\lambda_{i}\alpha_{i}F_{L}(u):~\lambda_{i}\geq0,~
\sum_{i=1}^3\lambda_{i}=1,~\alpha_i\in\R\Big\} \notag \\
&=&\{ \alpha F_{H}(u)+\beta F_{L}(u):~\alpha\geq \beta^{2},~\alpha,~\beta\in\R\}.
\label{co:100}
\end{eqnarray}
Thus
$${\mathcal C}({\rm co}~F(\R u))=\left\{ \alpha F_{H}(u)+
\beta F_{L}(u):~\alpha<\beta^{2},~\alpha,~\beta\in\R\right\},~~{\rm and~so}$$
\begin{eqnarray*}
\overline{{\mathcal C}({\rm co}~F(\R u))}&=&\{ \alpha F_{H}(u)+
\beta F_{L}(u):~\alpha\leq\beta^{2},~\alpha,~\beta\in\R\}=
F(\R u)+\R_+ d,
\end{eqnarray*}
since  $F(\R u)+\R_+ d=\{(\beta^2-t)F_H(u)+\beta F_L(u):~\beta\in\R,~t\geq 0\}$ by
Proposition \ref{co:00}.\\
$(b3)$: We write $F_L(u)=\lambda_1 F_H(u)+\lambda_2 d$ with
$\lambda_2\not=0$. By virtue of \eqref{eco:200}, we need to check that
$F(\R u)+\R_{+}F_H(u)+\R_+d\subseteq F(\R u)+\R_+d$. This requires to solve a
quadratic equation, which is always possible. Indeed, take $\alpha\in\R$,
$\lambda_+\geq 0$, $\gamma_+\geq 0$, we must find $\beta\in\R$ and $r_+>0$
such that
\begin{equation}\label{eco:201}
\alpha\lambda_2+\lambda_+=\beta\lambda_2+r_+,~~\alpha^2+\alpha\lambda_1+\gamma_+=
\beta^2+\beta\lambda_1.
\end{equation}
We can solve this system by substituting $\beta$ from the first equation of
\eqref{eco:201} into the second one, proving the convexity of
$F(\R u)+\R_+ d$.\\
Let us check the last assertion. By assumption, we can write
$d=\sigma_{1}F_{H}(u)+\sigma_{2}F_{L}(u)$ with $\sigma_{2}\neq0$.
From \eqref{co:100}, $x\in {\rm co}~F(\R u)$ if and only if
$x=\alpha^{2}F_{H}(u)+\beta F_{L}(u)$ with $\alpha^{2}\geq\beta^{2}$.
By taking $\gamma>0$ sufficiently large such that
$y\doteq F(tu)+\gamma d=[t^{2}+\sigma_{1}\gamma]u_{H}+
\left[t+\sigma_{2}\gamma\right]u_{L}$ with
$t^{2}+\sigma_{1}\gamma<\left(t+\sigma_{2}\gamma\right)^{2}$,
we get $y\in F(\R u)+\R_{+}d$ and $y\not\in {\rm co}~F(\R u)$.
\end{proof}

Next example shows that in fact $F(\R u)$ may be nonconvex for some $u$,
but it becomes convex once a particular direction is added.
\begin{ex} \label{ej:op0} Take
$$A=\begin{pmatrix}
0 & 1\cr
1 & 0
\end{pmatrix},~~
b=\begin{pmatrix}
1\cr
0
\end{pmatrix},
$$
with all other data vanish. Let $u=(u_1,u_2)$, $u_1u_2\not=0$. Then,
$F_H(u)=(2u_1u_2,0)$, $F_L(u)=(0,u_1)$ and
$F(\R u)=\{(x,y)\in\R^2:~x=\dfrac{2u_2}{u_1}y^2\}~~{\it is~ nonconvex},$
but certainly, $F(\R u)+\R_+d$ is convex if, and only if $d\not=(-2u_1u_2,0)$.
Here,
$$F(\R^2)=\{(0,0\}\cup [\R^2\setminus(\R\times\{0\})],~~F_H(\R^2)=
\R\times\{0\}.$$
\end{ex}

We note that, due to convexity,
\begin{equation}\label{eq:f:h}
\overline{F_H(\R^n)}=\R^2\Longleftrightarrow F_H(\R^n)=\R^2
\Longleftrightarrow {\rm int}~F_H(\R^n)=\R^2.
\end{equation}

As a consequence of the previous lemma we get the characterization of convexity.

\begin{thm}\label{teo:one}
Let $u\in \R^n$, $u\not=0$, and $f, g$ as above. Then,
\begin{itemize}
\item[$(a)$] $F(\R u)$ is convex $\Longleftrightarrow$ $\{F_H(u), F_L(u)\}$ is LD;
\item[$(b)$] in case $\{ F_{H}(u), F_{L}(u)\}$ is LI and $d\not=0$, one has
$$F(\R u)+\R_{+}d~~{\rm is~convex}~\Longleftrightarrow~-d\not\in\R_+F_H(u).$$
\end{itemize}
\end{thm}


\subsection{The case of higher dimension}

We first recall the following result due to Polyak:
\begin{thm} $\cite[Theorem~2.2]{pol1998}$ If $n\geq 2$ and there exist
$\alpha, \beta\in\R$ such that $\alpha A+\beta B\succ 0$, then
$F(\R^n)$ is convex (also closed).
\label{polayk00}
\end{thm}
Next theorem is an extension of the previous result. Indeed, Corollary 1 in  
\cite[page 498]{dines1941} establishes
$$\alpha A+\beta B\succ 0\Longleftrightarrow {\rm ND~holds~and}~F_H(\R^n)\not=\R^2.$$
Observe also that in case $F_H(\R^n)=\R^2$, one obtains $F(\R^n)=\R^2$ by 
Lemma  \ref{todoo}.

\begin{thm} \label{lema:connd}
Assume that ND holds for $A$ and $B$. Then
\begin{itemize}
\item[$(a)$] 
either $F_{H}(\mathbb{R}^{2})=\R^{2}$
or $F(\R^2)$ is convex$;$
\item[$(b)$] if $n\geq3$ then $F(\R^{n})$ is convex.
\end{itemize}
\end{thm}
\begin{proof} $(a)$: Assume that $F_{H}(\R^{2})\neq\R^{2}$. From
Proposition \ref{sev:impl} and Theorem \ref{thm:0}, we get SD for $A$ and $B$,
which means that there exist $\{ u, v\}$ LI satisfying $z_{u, v}=0$.
Thus $F(\R^{2})=F(\R u)+F(\R v)$. In addition, $F_H(u)\neq0\neq F_H(v)$ and
by the choice of $u$ and $v$, $F_H(u)\not=-\rho F_H(v)$ for all $\rho>0$. We claim
that
\begin{equation}\label{ecua100}
{\rm co}~F(\R u)+F(\R v)\subseteq F(\R u)+F(\R v).
\end{equation}
By virtue  of Lemma \ref{lema01}, we need only to consider
$\{F_{H}(u), F_L(u)\}$ to be LI. We can write for some
$\mu_i$ and $\sigma_i$, $i=1, 2$, $F_{H}(v)=\mu_{1}F_{H}(u)+\mu_{2}F_{L}(u)$ and
$F_{L}(v)=\sigma_{1}F_{H}(u)+\sigma_{2}F_{L}(u)$. Take any
$x\in {\rm co}~F(\R u)+F(\R v)$; then, by Lemma \ref{co:00},
$x=F(\alpha u)+\gamma^{2}F_{H}(u)+F(\beta v)$ for some $\alpha, \beta\in\R$
and $\gamma\in\R$.\\
We search for $\lambda_i\in\R$, $i=1, 2$ satisfying
$x=F(\lambda_{1}u)+F(\lambda_{2}v)$. From the last two equalities, we get
\[
\lambda_{1}^{2}+\lambda_{2}^{2}\mu_{1}+\lambda_{2}\sigma_{1}-
\alpha^{2}-\gamma^{2}-\beta^{2}\mu_{1}-\beta\sigma_{1}=0
\]
\[
\lambda_{1}+\lambda_{2}^{2}\mu_{2}+\lambda_{2}\sigma_{2}-
\alpha-\beta^{2}\mu_{2}-\beta\sigma_{2}=0
\]
From the second equation, we obtain
$\lambda_{1}=\alpha+\beta^{2}\mu_{2}+\beta\sigma_{2}-
\lambda_{2}^{2}\mu_{2}-\lambda_{2}\sigma_{2}$, which is substituted on the
left-hand side of the first equation to get a polynomial in $\lambda_2$, say
$p(\lambda_2)$. Our goal is to find a zero of $p$.
Observe that $\lambda_2=\beta$ implies $\lambda_1=\alpha$ and so
$p(\beta)=-\gamma^2\leq 0$. If $\mu_2\not=0$, the higher degree term of $p$ is
$\mu_2^2\lambda ^4$ which goes to $+\infty$ as $\lambda_2\to+\infty$; if
$\mu_2=0$, the higher degree term of $p$ is
$(\sigma_{2}^{2}+\mu_{1})\lambda_{2}^{2}$, with $\mu_1$ being positive
by the choice of $u$ and $v$. Thus, in both cases, $p(\lambda_2)>0$ for
$\lambda_2$ sufficiently large. Hence, there exists $p(\lambda_2)=0$, and so
\eqref{ecua100} is proved. We now check that ${\rm co}~F(\R^2)=F(\R^2)$.
Indeed, it is obtained from the following chain of equalities:
\begin{eqnarray*}
{\rm co}~F(\R^{2})&=&{\rm co}~F(\R u)+{\rm co}~F(\R v)=
{\rm co}~F(\R u)+F(\R v)+\R_{++}F_H(v)~~~~~~~~~\\
&=& F(\R u)+F(\R v)+\mathbb{R}_{++}F_{H}(v)=
F(\R u)+{\rm co}~F(\R v)\subseteq F(\R u)+F(\R v)\\
&=&F(\R^{2}).
\end{eqnarray*}
$(b)$: We will see now how we can reduce to the case $n=2$, so that $(a)$ is
applicable.\\
Let $x, y\in\R^n$ and $t\in~]0,1[$, we have
\begin{equation}\label{dim2:0}
tF(x)+(1-t)F(y)\in tF(\R x+\R y)+(1-t)(F(\R x+\R y).
\end{equation}
Thus, it suffices to prove the convexity of $F(\R x+\R y)$ whenever $\{x, y\}$
is LI. Take any $\lambda_i\in \R$, $i=1, 2$, then
$$f(\lambda_1 x+\lambda_2y)=\lambda_1^2\langle Ax,x\rangle+
2\lambda_1\lambda_2 \langle Ax,y\rangle+
\lambda_2^2 \langle Ay,y\rangle+\lambda_1 \langle a,x\rangle+
\lambda_2 \langle a,y\rangle.$$
$$=\begin{pmatrix}\lambda_1 & \lambda_2
\end{pmatrix}
\begin{pmatrix} \langle Ax,x\rangle & \langle Ax,y\rangle\cr
\langle Ax,y\rangle & \langle Ay,y\rangle
\end{pmatrix}
\begin{pmatrix}
\lambda_1\cr\lambda_2
\end{pmatrix}
+
\begin{pmatrix}
\langle a,x\rangle & \langle a,y \rangle
\end{pmatrix}
\begin{pmatrix}\lambda_1 \cr \lambda_2\end{pmatrix}.
$$
A similar expression is obtained for $g$. By denoting
$$\widetilde A(x,y)\doteq\begin{pmatrix}
\langle Ax,x\rangle & \langle Ax,y\rangle\cr
\langle Ax,y\rangle & \langle Ay,y\rangle
\end{pmatrix},~~
\widetilde B(x,y)\doteq\begin{pmatrix}
\langle Bx,x\rangle & \langle Bx,y\rangle\cr
\langle Bx,y\rangle & \langle By,y\rangle
\end{pmatrix},~~
$$
$$\widetilde a(x,y)\doteq\begin{pmatrix}
\langle a,x\rangle\cr
\langle a,y\rangle
\end{pmatrix},~~
\widetilde b(x,y)\doteq\begin{pmatrix}
\langle b,x\rangle\cr
\langle b,y\rangle
\end{pmatrix},
$$
we can write
\begin{equation}\label{n:dim_two}
F(\R x+\R y)
=\left\{\widetilde F(\lambda)\doteq
\begin{pmatrix} \langle \tilde A(x,y)\lambda,\lambda\rangle\cr
\langle \widetilde B(x,y)\lambda,\lambda\rangle\end{pmatrix}+
\begin{pmatrix} \langle\widetilde a(x,y),\lambda\rangle\cr
\langle\widetilde b(x,y),\lambda\rangle
\end{pmatrix}:~~\lambda\in\R^2\right\}=\widetilde F(\R^2).
\end{equation}
We want to apply the result in $(a)$ to the set on the right-hand side of
\eqref{n:dim_two}. It is not difficult to verify that if ND holds for
$A$ and $B$, then ND also holds for $\widetilde A(x,y)$ and $\widetilde B(x,y)$
provided $\{x, y\}$ is LI. Furthermore, since $F_H(\R^n)\not=\R^2$, we get
$\widetilde F_H(\R^2)\not=\R^2$. By applying $(a)$, we conclude
that $\widetilde F(\R^2)=F(\R x+\R y)$ is convex, and therefore the convexity
of $F(\R^n)$.
\end{proof}

In order to establish our second main result without ND, some preliminaries are
needed.

\begin{prop}\label{non:nd}
Let $n\geq 2$ and $0\not=v\in\R^n$ such that $F_H(v)=0$. The following assertions
hold:
\begin{itemize}
\item[$(a)$] $F_H(\R^2)\not=\R^2$ and $\{Av, Bv\}$ is LD$;$
\item[$(b)$] if $n\geq 3$ then either $F_H(\R^n)=\R^2$ or
$\{Av, Bv\}$ is LD.
\item[$(c)$] The set $Z\doteq\{z_{u,v}:~u\in\R^n\}$ is a vector subspace, and if
$F_H(\R^n)\not=\R^2$ then for all $u\in\R^n$ satisfying $z_{u,v}\not=0$,
\begin{equation}\label{bd:00}
{\rm bd}~F_H(\R^n)=\R z_{u,v}.
\end{equation}
In particular, if $Av\not=0$ and $Bv=\lambda Av$ $($resp. $Bv\not=0$ and
$Av=\lambda Bv)$ for some $\lambda\in\R$, then
\begin{equation}\label{bd:0}
{\rm bd}~F_H(\R^n)=\R(1,\lambda)~~~(resp.~~
{\rm bd}~F_H(\R^n)=\R(\lambda,1)).
\end{equation}
\end{itemize}
\end{prop}
\begin{proof} $(a)$: The first part follows from Proposition \ref{sev:impl}.
By assumption $\{Av, Bv\}\subseteq v^\perp$, thus
$\{Av, Bv\}$ is LD.\\
$(b)$: Again $\{Av, Bv\}\subseteq v^\perp$. Let $x, y\in\R^n$ be LI vectors.
We consider first the case where $\{z_{v,x}, z_{v,y}\}$  is LD. In this case
there exist $\lambda_1, \lambda_2\in\R$ not both null such that
$\lambda_1 z_{v,x}+\lambda_2 z_{v,y}=0$. The latter means
$z_{v,\lambda_1 x+\lambda_2 y}=0$, which implies that
$$\{Av, Bv\}\subseteq [{\rm span}\{v, \lambda_1 x+\lambda_2 y\}]^\perp.$$
The latter subspace has dimension $n-2$. If $n-2$ equals 1, we are done; if
$n-2\geq 2$, we proceed in the same manner until reaching dimension 1, in
which case we conclude that $\{Av, Bv\}$ is LD.\\
Now consider the case where $\{z_{v,x}, z_{v,y}\}$  is LI. Take any $w\in\R^2$
and write $w=\alpha z_{v,x}+\beta z_{v,y}$ for some $\alpha, \beta\in\R$. We
easily obtain for all $\varepsilon>0$:
$$F_H(\dfrac{1}{\varepsilon}x+\alpha\dfrac{\varepsilon}{2}v)=
\dfrac{1}{\varepsilon^2}F_H(x)+\alpha z_{v,x},~
F_H(\dfrac{1}{\varepsilon}y+\beta\dfrac{\varepsilon}{2}v)=
\dfrac{1}{\varepsilon^2}F_H(y)+\beta z_{v,y}.$$
By Dines theorem $F_H(\R^n)$ is a convex cone, therefore
$$\dfrac{1}{\varepsilon^2}(F_H(x)+F_H(y))+w\in F_H(\R^n),~~\forall~
\varepsilon>0.$$
Letting $\varepsilon\to+\infty$, we get $w\in\overline{F_H(\R^n)}$, proving that
$\overline{F_H(\R^n)}=\R^2$, and the result follows from \eqref{eq:f:h}.\\
$(c)$:   Obviously $Z$ is a vector subspace. Let $u\in\R^n$, $z_{u,v}\not=0$.
$F_H(u\pm tv)=F_H(u)\pm2tz_{u,v}$ for all $t\in\R$, which implies that
$\pm z_{u,v}\in \overline{F_H(\R^n)}$. Since the latter set is a convex cone
different from $\R^2$, $\overline{F_H(\R^n)}$ is either a halfspace or the
straightline $\R z_{u,v}$. In either case we obtain \eqref{bd:00}.\\
For the last part simply observe that $z_{Av, v}=\|Av\|^2(1,\lambda)\not=(0,0)$.
\end{proof}

When ND does not hold, next result asserts the convexity of $F(\R^2)$
under nonemptiness of the interior of the homogeneous part.
\begin{lem} \label{open:1} Assume that ND does not hold. If
${\rm int}~F_H(\R^2)\not=\emptyset$, then $F(\R^2)$ is convex.
\end{lem}
\begin{proof} Let $v\not=0$ satisfying $F_H(v)=0$. Take $u\in\R^2$ such that
$\{u, v\}$ is LI. It follows that $F_H(\R^2)=F_H(\R u+\R v)=\R_+F_H(u)+\R z_{u,v}$.
Since ${\rm int}~F_H(\R^2)\not=\emptyset$, one gets $\{F_H(u), z_{u,v}\}$ is LI.
We will check that $F(\R^{2})=F(\R u+\R v)$ is convex. From Theorem 2 in
\cite{rg1995}, it suffices to prove that
$F(\R u+\R v)=F_H(\R u+\R v)+ F(\R u+\R v)$.\\
Let $\alpha,\,\beta,\,\gamma,\delta\in\R$ and $s=(s_u,s_v),\, h=(h_u,h_v)\in\R^2$
such that $F_{L}(u)=s_{u}z_{u,\, v}+h_{u}F_{H}(u)$ and
$F_{L}(v)=s_{v}z_{u,\, v}+h_{v}F_{H}(u)$. We must find
$\lambda=(\lambda_1,\lambda_2)\in\R^2$ satisfying
$F(\alpha u+\beta v)+F_{H}(\gamma u+\delta v)=F(\lambda_{1}u+\lambda_{2}v)$.
This equality along with the LI of $\{ F_{H}(u), z_{u,v}\}$ lead to the following
two equations:
\begin{eqnarray*}
2\lambda_{1}\lambda_{2}-2\left(\alpha\beta+\gamma\delta\right)+
s_{u}\left(\lambda_{1}-\alpha\right)+s_{v}\left(\lambda_{2}-\beta\right)&=&0\\
\lambda_{2}^{2}-\left(\beta^{2}+\delta^{2}\right)+
h_{u}\left(\lambda_{1}-\alpha\right)+h_{v}\left(\lambda_{2}-\beta\right)&=&0
\end{eqnarray*}
If $h_u\neq 0$, from the second equation we get a expression for $\lambda_{1}$ and
substitutes it into the first one. The obtained equation is polynomial of
third degree in the variable $\lambda_2$, so it admits at least one real zero. Thus
a solution $(\lambda_1,\lambda_2)$ of the above equations is found. \\
Now consider the case $h_{u}=0$. The second equation is quadratic in $\lambda_{2}$
with discriminat $\Delta=\left(h_{v}+2\beta\right)^{2}+4\delta^{2}\geq0$.
Thus the second equation always admits a solution $\lambda_2\in\R$. Since the first
equation is linear in $\lambda_1$, it will be solvable in $\lambda_1$ provided its
coeficient $2\lambda_2+s_u$ is non zero. This is satisfied if $\Delta>0$. If
$\Delta =0$ and the worst case $2\lambda_{2}+s_{u}=0$ fulfills, we easily see
that the first equation is satisfied vacuously.
\end{proof}

In what follows, in view of
$$F(\R^{n})=F(({\rm ker}~A\cap {\rm ker}~B)^\perp)+
F_L({\rm ker}~A\cap {\rm ker}~B),$$
we show that there is no loss of generality in assuming
${\rm ker}~A\cap{\rm ker}~B=\{0\}$. In fact, set $K\doteq
{\rm ker~A}\cap{\rm ker}~B$ with ${\rm dim}~K^\perp=m$. Take a basis
$\{u_i:~1\leq i\leq m\}$ of $K^\perp$. Thus
$F(K^\perp)=\widetilde{F}(\R^{m})$ is a pair of quadratic functions having the
following data: $\widetilde{A}=(\langle u_{i},Au_{j}\rangle)_{ij}$,
$\widetilde{B}=(\langle u_{i},Bu_{j}\rangle)_{ij}$, $\widetilde{a}=(
\langle a,u_{1}\rangle,\ldots, \langle a,u_m\rangle)$
and  $\widetilde{b}=(\langle b,u_{1}\rangle,\ldots, \langle b,u_m\rangle)$.
Let us prove that
$\widetilde K\doteq {\rm ker}~\widetilde A\cap{\rm ker}~\widetilde B=\{0\}$.
Take $z\in \widetilde{K}$. Then
$$\langle u_{i},A(\sum_{j=1}^m z_{j}u_{j})\rangle=0~~{\rm and}~~
\langle u_{i},B (\sum_{j=1}^m z_{j}u_{j})\rangle=0~~~\forall~
i=1,\ldots , m.$$
This means $\displaystyle \sum_{j=1}^m z_{j}u_{j}\in
K^{\bot}\cap K^{\bot\bot}=\{ 0\} $, and so
$\widetilde{K}=\{ 0\} $. This condition will be assumed in $(b)$ of the following
lemma, which is the second main Dines-type result without ND property.

\begin{lem} \label{open:2}
The set $F(\R^{n})$ is convex under any of the following conditions:
\begin{itemize}
\item[$(a)$] $F_L({\rm ker}~A\cap{\rm ker}~B)\not=\{0\};$
\item[$(b)$] $\emptyset\neq {\rm int}~F_{H}(\R^{n})\neq\R^{2}$.
\end{itemize}
\end{lem}
\begin{proof} $(a)$ Let $u\in {\rm ker}~A\cap{\rm ker}~B$ and set
$0\not=h\doteq F_L(u)$. Then, for all $x\in\R^n$,
$F(x+tu)=F(x)+th\in F(\R^n)$. Lemma \ref{denso} yields the desired result.\\
$(b)$: We apply the procedure as above to consider $F(K^\perp)=\widetilde F(\R^m)$
and $F(\R^{n})=\widetilde F(\R^m)+F_{L}(K)$ with
$K\doteq{\rm ker}~A\cap{\rm ker}~B$.
Obviously ${\rm dim}~ K^{\perp}=m\geq2$ since
$\emptyset\neq {\rm int}~F_{H}(\R^{n})={\rm int}~\widetilde{F}_{H}(\R^{m})$. Then,
if ND holds for $\widetilde A$ and $\widetilde B$, by Lemma \ref{lema:connd},
$\widetilde F(\R^m)$ is convex, and so of $F(\R^n)$ as well. In case ND does not
hold, we proceed on $F$, by assuming now that
${\rm ker}~A\cap{\rm ker}~B=\{0\}$.\\
Let $v\neq0$ satisfying $F_{H}(v)=0$. It is not difficult to check that
$\{z_{u, v}:~ u\in\R^{n}\}$ is contained in a line passing through the origin;
actually it is the entire line since $z_{-u,v}=-z_{u,v}$ and
${\rm ker}~A\cap{\rm ker}~B=\{ 0\}$. Thus, $\{ z_{u, v}:~u\in\R^{n}\}=\R\pi$ with
$\pi\neq0$. Let us define
$$X_v\doteq\{ u\in\R^{n}:~z_{u, v}\neq0\},~~Y_v\doteq
\{ u\in\R^{n}:~F_{H}(u)\not\in\R\pi\},$$
and consider $C_v\doteq X_v\cap Y_v$. Besides $X_{v}$ is nonempty since
${\rm ker}~A\cap {\rm ker}~B=\{ 0\}$, it is also dense ($u_{0}\in X_{v}$ implies
$u+\dfrac{1}{k}u_{0}\in X_{v}$ for any $u\in\R^{n}$ and $k\in\N$); $Y_{v}$ is
nonempty in view of $\emptyset\neq{\rm int}~F_{H}(\R^{n})$, and open by continuity.
It is also dense (take $u_{0}\in Y_{v}$  and note that for every
$u\notin Y_{v}$, one has $\langle \pi_{\bot},F_{H}(u+\dfrac{1}{k}u_{0})\rangle=
\dfrac{2}{k}\langle \pi_{\bot},z_{u, u_{0}}\rangle+
\dfrac{1}{k^{2}}\langle \pi_\bot, F_{H}(u_{0})\rangle\neq0$ for all $k\in\N$
sufficiently large, implying $u+\dfrac{1}{k}u_{0}\in Y_{v}$ for all $k\in\N$)
sufficiently large. Consequently, $C_v$ is nonempty and dense since it is the
intersection of two dense sets being one of them open.
Notice that for all $u\in C_v$, $\{u, v\}$ is LI and therefore $F(\R u+\R v)$
satisfies all the assumptions of Lemma \ref{open:1}, so it is convex. Moreover
$$F_{H}(\R u+\R v)=\{ 0\} \cup \{ \R\pi+\R_{++}F_{H}(u)\}=
\{ 0\} \cup\{ h\in\R^{2}:~\langle \pi_{\bot},F_{H}(u)\rangle
\langle \pi_{\bot}, h\rangle>0\},$$
where the second equality follows from Proposition \ref{felo0}.
On the other hand, all the elements of the form
$\langle\pi_{\bot},F_{H}(u)\rangle$ have the same sign since
$F_{H}(\R^{n})\neq\R^{2}$. Hence, by using Theorem 2 in \cite{rg1995}, we
obtain
\begin{eqnarray*}
F(\R^{n})&\supseteq& F(\R u+\R v)= F(\R u+\R v)+F_{H}(\R u+\R v)\\
&\supseteq& F(\R u+\R v)+\R\pi+\R_{++}F_{H}(u)\\
&=&F(\R u+\R v)+\{ h\in\R^{2}:~ r\langle\pi_{\bot},h\rangle>0\}
\end{eqnarray*}
with $F_{H}(u)\not\in\R\pi$ for all $u\in C_v$ and some constant $r\not=0$.
Thus, by Lemma \ref{denso} (with $X=\R C_v+\R v$), $F(\R^{n})$ is convex.
\end{proof}

Next lemma is also new in the literature.
\begin{lem}\label{todoo}
Assume that $F_{H}(\R^{n})=\R^{2}$. Then $n\geq 2$ and
$F(\R^{n})=\R^{2}$.
\end{lem}
\begin{proof} The fact that $n\geq 2$ is obvious. Consider first $n=2$ and let
$L_{1}\in\R^{2}$ be any non-zero vector. Take $u$ and $v$ satisfying
$F_{H}(u)=-F_H(v)=L_{1}$. Thus $\{u, v\}$ is LI. Since $F_{H}(\R^{n})=\R^{2}$,
$\{z_{u,\, v}, L_1\}$ is LI. Set $L_{2}\doteq z_{u, v}$. Then, there exist
$\sigma_i$, $\rho_i$, $i=1, 2$, such that $F_{L}(u)=2\sigma_{1}L_{1}+\rho_{1}L_{2}$
and $F_{L}(v)=-2\sigma_{2}L_{1}+\rho_{2}L_{2}$. Given any $x\in\R^2$, we will find
$\lambda_i\in\R$, $i=1, 2$, satisfying
\begin{equation}\label{dos:0}
x=F((\lambda_{1}-\sigma_{1})u+(\lambda_{2}-\sigma_{2})v).
\end{equation}
On the other hand, we obtain
$$F((\lambda_{1}-\sigma_{1})u+(\lambda_{2}-\sigma_{2})v)=
[\lambda_{1}^{2}-\lambda_{2}^{2}]L_{1}+[2\lambda_{1}\lambda_{2}+
(\rho_{1}-\sigma_{2})\lambda_{1}+(\rho_{2}-\sigma_{1})\lambda_{2}]L_{2}-C,$$
with $C=(\sigma_{1}^{2}-\sigma_{2}^{2})L_{1}+
(\sigma_{1}\rho_{1}+\sigma_{2}\rho_{2})L_{2}$. By writing
$x=x_{1}L_{1}+x_{2}L_{2}-C$ and setting $\pi_1=\rho_{1}-\sigma_{2}$,
$\pi_2=\rho_{2}-\sigma_{1}$, \eqref{dos:0} yields
\begin{equation}\label{dos:1}
x_{1}=\lambda_{1}^{2}-\lambda_{2}^{2},~~~
x_{2}=2\lambda_{1}\lambda_{2}+\pi_{1}\lambda_{1}+\pi_{2}\lambda_{2}.
\end{equation}
We distinguish two cases.\\
Suppose first that the set $\{(2,\pi_1),( \pi_2,-x_2)\}$
is LD. Then, there exists $t_{0}$ such that
$t_0(2,\pi_1)=( \pi_2,-x_2)$.  Thus, the second equation in \eqref{dos:1} reduces
to $0=(\lambda_{1}+t_{0})(2\lambda_{2}+\pi_{1})$. If $0=\lambda_{1}+t_{0}$
then $x_{1}=t_{0}^{2}-\lambda_{2}^{2}$ for any $\lambda_{2}\in\R$; if
$0=2\lambda_{2}+\pi_{1}$ then
$x_{1}=\lambda_{1}^{2}-(\dfrac{\pi_{1}}{2})^{2}$
for any  $\lambda_{1}\in\R$. From this we infer that the first equation in
\eqref{dos:1} is always satisfied as well.\\
Suppose now that the set $\{(2,\pi_1),( \pi_2,-x_2)\}$ is LI, which is equivalent
to $2x_2+\pi_1\pi_2\not=0$ by Proposition  \ref{felo0}. From the second
equation in \eqref{dos:1}, we obtain, by assuming additionally 
$2\lambda_{2}+\pi_{1}\not=0$ (since otherwise we are done)
$$\lambda_{1}=\dfrac{x_{2}-\pi_{2}\lambda_{2}}{2\lambda_{2}+\pi_{1}}=
-\dfrac{\pi_{2}}{2}+\dfrac{2x_{2}+\pi_{1}\pi_{2}}{2(2\lambda_{2}+\pi_{1})}.$$
Thus,
$$x_{1}=\lambda_{1}^{2}-\lambda_{2}^{2}=
\left(-\dfrac{\pi_{2}}{2}+\dfrac{2x_{2}+
\pi_{1}\pi_{2}}{2(2\lambda_{2}+\pi_{1})}\right)^{2}-\lambda_{2}^{2}\doteq
p(\lambda_{2}).$$
Since $p(]-\infty,-\dfrac{\pi_{1}}{2}[)=\R=
p(]-\dfrac{\pi_{1}}{2}, +\infty[)$, we conclude that system \eqref{dos:1} admits
a solution, proving that $F(\R^2)=\R^{2}$.\\
We consider now that $n\geq 3$. Take any $u$ and $v$ satisfying
$F_{H}(u)=(1, 0)$ and $F_{H}(v)=(0, 1)$. Then
$\R_{+}^{2}\subseteq F_{H}(\R u+\R v)$, which implies
${\rm int}~F_{H}(\R u+\R v)\not=\emptyset$. In case
$F_{H}(\R u+\R v)=\R^{2}$, we apply the above result to conclude that
$\R^{2}=F(\R u+\R v)$ and therefore $F(\R^{n})=\R^2$. If on the contrary,
$F_{H}(\R u+\R v)\not=\R^{2}$, from Lemma \ref{open:2}, we get the convexity of
$F(\R u+\R v))$. By Theorem 2 in \cite{rg1995},
$\R_{+}^{2}\subseteq F(\R u+\R v)+F_{H}(\R u+\R v)=
F(\R u+\R v)\subseteq F(\R^{n})$. Similarly, we also get the sets
$-\R_{+}^{2}$, $\R_{+}\times\R_{-}$ and $\R_{-}\times\R_{+}$ are contained
in $F(\R^{n})$, and therefore $F(\R^n)=\R^{2}$.
\end{proof}

By using the previous two lemmas and Theorem \ref{lema:connd},
the following theorem is obtained.

\begin{thm}\label{int:dif:0} Let $n\geq 2$. If either 
${\rm int}~F_H(\R^n)\not=\emptyset$ or ND holds for $A$ and $B$ then 
$F(\R^n)$ is convex.
\end{thm}

We now describe a procedure to find a suitable change of variable to be used
presently.

\begin{lem} \label{cambio:0} Let $d=(d_{1}, d_{2})\neq0$, and consider
$F$ as in \eqref{bi:quad0} with $A=d_{1}I$ and $B=d_{2}I$ with $I$ being the
identity matrix of order $n$ and $a, b$ any vectors in $\R^n$.
Then, there exist $t_{0}\geq0$, $k\in\R^{2}$, $\bar x\in\R^{2}$ and a square
matrix $C$  satisfying $C^\top C=I$ such that, if $x=Cy-\bar x$ one obtains
\begin{itemize}
\item[$(a)$] $F(x)=\widetilde{F}(y)-k$ where $\widetilde F$ is defined in terms
of $\widetilde{A}=A$, $\widetilde{B}=B$, $\widetilde{a}=-d_{2}t_{0}e_{1}$,
$\widetilde{b}=d_{1}t_{0}e_{1}$ with $e_{1}=(1,0,\ldots,0)\in\R^n;$
\item[$(b)$] if $n\geq2$ then $F(\R^{n})=\widetilde{F}(\R e_{1})+\R_{+}d-k=
{\rm co}~\widetilde{F}(\R e_{1})-k$, and there exists $\bar y\in\R^n$
with $\widetilde{F}_{H}(\bar y)=d$ and $\widetilde{F}_{L}(\bar y)=0;$
\item[$(c)$] the following statements are equivalent (or both fail if $n\geq2):$
\begin{enumerate}
\item[$(c1)$] $\{d, F_{L}(x)\}$ is LI for all $x\in F_H^{-1}(d);$
\item[$(c2)$] $\{d, \widetilde{F}_{L}(y)\}$ is LI for all
$y\in\widetilde{F}_{H}^{-1}(d)$.
\end{enumerate}
\end{itemize}
\end{lem}
\begin{proof} $(a)$: As $d$ and $d_{\bot}$ are LI, there exist
$\bar x,~\bar y\in\R^{n}$ satisfying $\left(\begin{array}{c}
a_{i}\\
b_{i}
\end{array}\right)=2\bar x_{i}d+\bar y_{i}d_{\bot}$ for all $i$. For any $x\in\R^n$,
we write
\begin{eqnarray*}
F(x)&=&\langle x,x\rangle d+2\langle\bar x,x\rangle d+\langle\bar y,x
\rangle d_{\bot}\\
&=&\langle x+\bar x,x+\bar x\rangle d+\langle\bar y,x+\bar x\rangle d_\bot-
\langle\bar x,\bar x\rangle d-\langle \bar x,\bar y\rangle d_{\bot}.
\end{eqnarray*}
If $\bar y=0$, we choose $t_{0}=0$, $C=I$ and the conclusion follows; otherwise
take $x+\bar x=Cy$ with $C=\left(\begin{array}{cc}
\dfrac{\bar y}{\|\bar y\|} & W\end{array}\right)$ where $W$ is any matrix
having as columns a ortonormal basis of $\bar y^\perp$.
Clearly $C^{\top}C=I$ and, by choosing $t_{0}=\|\bar y\|$, we get
\begin{equation}\label{f:to:tildef}
F(x)=\widetilde F(y)-k,~~\widetilde F(y)=\langle y,y\rangle d+
t_0y_1 d_{\bot},~~k\doteq \|\bar x\|^2d+\langle\bar x,\bar y\rangle d_{\bot}.
\end{equation}
$(b)$: From the last equality, we obtain
$$F(x)=y_{1}^{2}d+\|\bar y\|y_{1} d_\bot+d\sum_{i\geq2}y_{i}^{2}-k,$$
which implies that $F(\R^{n})=\widetilde{F}(\R e_{1})+\R_{+}d-k$; the second
equality in $(b)$ follows from Proposition  \ref{co:00} since
$\widetilde{F}_H(e_1)=d$. In addition, we obtain
$\widetilde{F}_{H}(e_{2})=d$ and $\widetilde{F}_{L}(e_{2})=0$.\\
$(c1)\Rightarrow(c2)$: From above we deduce
$$F(x)=F(Cy-\bar x)=F(Cy)+F(-\bar x)-2z_{Cy,\bar x}=
\widetilde{F}(y)-k,$$
with $k=-F(-\bar x)$, $\widetilde{F}_{H}(y)=F_{H}(y)=F_{H}(Cy)$ and
$\widetilde{F}_{L}(y)=F_{L}(Cy)-2z_{Cy,\bar x}$.
Let $y\in\widetilde{F}_{H}^{-1}(d)$. Then $F_{H}(Cy)=d$, and by $(c1)$
$\{F_{L}(Cy), d\}$ is LI. Thus $\{\widetilde{F}_{L}(y), d\}$ is LI as well, since
$\widetilde{F}_L(y)=F_{L}(Cy)-2z_{Cy,\bar x}$ and
$z_{Cy,\bar x}=\left(\begin{array}{c}
\langle \bar x,ACy\rangle\\
\langle \bar x,BCy\rangle
\end{array}\right)=\langle\bar x,Cy\rangle d$.\\
$(c2)\Rightarrow(c1)$: it is similar.\\
In case $n\geq 2$, both expressions $(c1)$ and $(c2)$ fail in view of $(b)$.
\end{proof}

Next theorem characterizes those directions $d$ under which $F(\R^n)+\R_+d$
is convex.
\begin{thm} \label{teor:1000} Let $f, g$ be any quadratic functions as above
and $d=(d_1,d_2)\in\R^2$, $d\not=0$. The following assertions are equivalent:
\begin{itemize}
\item[$(a)$] $F(\R^n)+\R_{+}d$ is nonconvex;
\item[$(b)$] The following hold:
\begin{enumerate}
\item[$(b1)$] $F_L({\rm ker}~A\cap{\rm ker}~B)=\{0\};$
\item[$(b2)$] $d_2A=d_1 B;$
\item[$(b3)$] $F_H^{-1}(-d)\not=\emptyset$ and $\{d,F_L(u)\}$ is LI for all
$u\in F_H^{-1}(-d).$
\end{enumerate}
\end{itemize}
\end{thm}
\begin{proof} $(a)\Rightarrow(b)$: From Lemma \ref{open:2}, we get
$F_L({\rm ker}~A\cap{\rm ker}~B)=\{0\}$ and so $(b1)$ holds, and additionally
${\rm int}~F_H(\R^n)=\emptyset$. We now introduce the function $\widetilde F$
which has the same form as $F$, but on $\R^{n+1}$, with
$\widetilde F(0)=0$ and  data
$$\widetilde{A}=\left(\begin{array}{cc}
A & 0\\
0 & d_{1}
\end{array}\right),~ \widetilde{B}=\left(\begin{array}{cc}
B & 0\\
0 & d_{2}
\end{array}\right),~\widetilde{a}=\left(\begin{array}{c}
a\\
0
\end{array}\right),~\widetilde{b}=\left(\begin{array}{c}
b\\
0
\end{array}\right).$$
Then, we get
$\widetilde F(\R^{n+1})=F(\R^n)+\R_+ d$ and $\widetilde F_H(\R^{n+1})=
F_H(\R^n)+\R_+ d$.
Since $(b2)$ holds if and only if $F_H(\R^n)\subseteq\R d$, one gets
${\rm int}~\widetilde F(\R^{n+1})\not=\emptyset$ if $F_H(\R^n)\not\subseteq\R d$.
Hence, if $(b2)$ does not hold $\widetilde F(\R^{n+1})$ is convex by Lemma
\ref{open:2}, that is, $F(\R^n)+\R_+ d$ is convex, proving $(a)$ implies
$(b2)$.\\
We now check that $F_H^{-1}(-d)\not=\emptyset$. If on the contrary
$-d\not\in F_H(\R^n)$, we get $F_H(\R^n)\subseteq \R_+ d$  by $(b2)$.
Thus, either $F_H(\R^n)=\{0\}$ or
$F_H(\R^n)=\R_+ d$. In the first case $A=0$ and $B=0$, implying the
convexity of $F(\R^n)$, which is not possible if $(a)$ is assumed. The second case
is also impossible due to $(c)$ of Proposition \ref{non:nd}, proving
the first part of $(b3)$. Let us prove the second part
of $(b3)$. Take $u\in\R^n$ such that $F_{H}(u)=-d$ and $F_{L}(u)=\lambda_{0}d$
for some $\lambda_0\in\R$.
From $(b2)$ for all $x\in\R^{n}$, $z_{x, u}\in\R d$. This along with
the fact that $F(x+tu)=F(x)+2z_{x, u}-t^{2}d+t\lambda_{0}d$, yield
$F(x)+\R d\in F(\R^{n})$ for all $x\in\R^n$. Thus, the convexity of
$F(\R^{n})$ follows from Lemma \ref{denso}, which contradicts $(a)$.\\
$(b)\Rightarrow(a)$: By a spectral theorem, we can find a non singular matrix
$D$ satisfying $D^\top AD=d_{1}\left(\begin{array}{cc}
I_{m_{1}} & 0\\
0 & -I_{m_{2}}
\end{array}\right)$, where $I_{l}$ denotes the identity matrix or order $l$ (in
view of $(b1)$ we may ignore the null eigenvalues if any), and
$m_{2}\geq1$ by $(b3)$. From $(b2)$,
we also get $D^{T}BD=d_{2}\left(\begin{array}{cc}
I_{m_{1}} & 0\\
0 & -I_{m_{2}}
\end{array}\right)$.\\
We apply the preceding lemma to both blocks corresponding to the matrices $A$ and
$B$. Thus, we obtain $m_{2}=1$ since otherwise $(b3)$ would be impossible by
virtue of $(c)$ in Lemma \ref{cambio:0}. Hence we may assume from now on
$m_{1}=m$ and $m_{2}=1$. From Lemma \ref{cambio:0}, there exist $0\leq t_{1}$,
$t_{2}$, $k\in\R^{2}$, $\bar x\in\R^{n}$  and a square matrix
$C=\left(\begin{array}{cc}
C_{m_{1}} & 0\\
0 & C_{m_{2}}
\end{array}\right)$
such that $C^\top C=I$, and if $x=Cy-\bar x$, one obtains
$F(x)=\widetilde{F}(y)-k$, where $\widetilde F$ is as $F$ with data
$\widetilde{A}=D^\top AD$,
$\widetilde{B}=D^\top B D$,
$\widetilde{a}=-d_{2}t_{1}e_{1}-d_{2}t_{2}e_{m+1}$, and
$\widetilde{b}=d_{1}t_{1}e_{1}+d_{1}t_{2}e_{m+1}$.\\
By $(b3)$, $\{d, \widetilde{F}_{L}(y)\}$ is LI for all
$y\in\widetilde{F}_{H}^{-1}(-d)$. This last expression means
$\widetilde{F}_{H}(y)=(\sum_{i=1}^my_{i}^{2}-y_{m+1}^{2})d=-d$, which reduces to
$y_{m+1}^{2}=1+\sum_{i=1}^m y_{i}^{2}$. \\
On the other hand, $\widetilde{F}_{L}(y)=(t_{1}y_{1}+t_{2}y_{m+1})d_{\bot}$. Then
$\{d, \widetilde{F}_{L}(y)\}$ is LI if and only if $t_{1}y_{1}+t_{2}y_{m+1}\not=0$.

Let us show that $F(\R^{n})+\R_{+}d$ is nonconvex. First, observe that
$\widetilde{F}(\pm\gamma e_{m+1})=-\gamma^{2}d\pm\gamma t_{1}d_{\bot}$,
so that $-\gamma^{2}d\pm\gamma t_{1}d_{\bot}-k\in F(\R^{n})$ for all $\gamma>0$.
We now check that for all $\gamma>0$,
$$-\gamma^{2}d-k=
\cfrac{-\gamma^{2}d+\gamma t_{1}d_{\bot}-k-\gamma^{2}d-
\gamma t_{1}d_{\bot}-k}{2}\not\in F(\R^{n}),$$
which turn out that $-\gamma^{2}d-k\notin F(\R^{n})+\R_{+}d$
for all $\gamma$ sufficiently large. Assume that there exists $y\in\R^{m+1}$ such
that $-\gamma^{2}d=\widetilde{F}(y)$. But $\widetilde{F}(y)=
\left(\overset{m}{\underset{i=1}{\sum}}y_{i}^{2}-y_{m+1}^{2}\right)d+
(t_{1}y_{1}+t_{2}y_{m+1})d_{\bot}$, so
\begin{equation}\label{ecua:100}
y_{m+1}^{2}=\gamma^{2}+\overset{m}{\underset{i=1}{\sum}}y_{i}^{2}~~{\rm and}~~
t_{1}y_{1}+t_{2}y_{m+1}=0.
\end{equation}
This yield a contradiction, since the second equality implies that
$\widetilde F(\dfrac{1}{\gamma}y)=\widetilde F_H(\dfrac{1}{\gamma}y)=-d$ and
therefore $\{d,\widetilde F_L(\dfrac{1}{\gamma}y)\}$ must be LI,
that is, as observed above,
$t_{1}y_{1}+t_{2}y_{m+1}\not=0$.
\end{proof}

From the preceding result the following theorem follows.

\begin{thm}\label{todo:0}
Let $n\geq 1$ and $f, g$ be any quadratic functions as above.
If $F(\R^n)+\R_+d$ is convex for all $d\in\R^2$, $d\not=0$, then
$F(\R^n)$ is convex.
\end{thm}
\begin{proof} If $F(\R^n)$ is nonconvex then $F_H(\R^n)\not=\{0\}$, and
by Lemma \ref{open:2}, $F_L({\rm ker}~A\cap {\rm ker}~B)=\{0\}$ and
${\rm int}~F_H(\R^n)=\emptyset$.
From the latter condition, $F_H(\R^n)\subseteq \R d$ for some $d\in\R^2$,  which is
equivalent, as seen in the proof of the previous theorem, to $d_2A=d_1B$.
Actually either $F_H(\R^n)=\R d$ or $F_H(\R^n)=\R_+ d$ or $F_H(\R^n)=-\R_+ d$.
In case $F_H^{-1}(-d)\not=\emptyset$, we proceed as follows. By Theorem
\ref{teor:1000}, $\{d,F_L(u)\}$ is LD for some (all) $u\in F_H^{-1}(-d)$. Then,
for such $u$, $F_L(u)=\gamma d$ for some $\gamma\in\R$. On the other hand, for
all $x\in\R^n$, all $t\in\R$, assuming $d_2\not=0$, one has
$$F(x+d_2tu)=F(x)-d_2^2t^2d+\gamma d_2td+2td_2z_{x,u}=
F(x)-d_2^2t^2d+\gamma d_2 td+2t\langle Bx,u\rangle d$$
(In case $d_1\not=0$, one has $F(x+d_1tu)=
F(x)-d_1^2t^2d+\gamma d_1 td+2t\langle Ax,u\rangle d$).
From  this, we infer $F(\R^n)-\R_+d\subseteq F(\R^n)$. If $F_H^{-1}(-d)=\emptyset$
but $F_H^{-1}(d)\not=\emptyset$, we work with $\tilde d=-d$ to conclude with
the same equality as above, implying $F(\R^n)-\R_+\tilde d\subseteq F(\R^n)$.
Thus $F(\R^n)+\R_+d\subseteq F(\R^n)$. The previous reasoning proves, in any
of the three situations for $F_H(\R^n)$, that $F(\R^n)+F_H(\R^n)\subseteq F(\R^n)$.
Hence, $F(\R^n)$ is convex as a consequence of Theorem 2 in \cite{rg1995}, so a
contradiction is reached, establishing that in fact $F(\R^n)$ is convex.
\end{proof}

By combining the last two theorems, we obtain the next result which characterizes
the convexity of joint-range for a pair of quadratic functions.
 \begin{thm}\label{conv:char}
Let $n\geq 1$ and $f, g$ be any quadratic functions as above. Then,
$F(\R^n)$ is convex if, and only if for all $d\in\R^2$, $d\not=0$, any of the 
following conditions hold:
\begin{itemize}
\item[$(C1)$] $F_L({\rm ker}~A\cap{\rm ker}~B)\not=\{0\};$
\item[$(C2)$] $d_1B\not=d_2A;$
\item[$(C3)$] $F_H^{-1}(-d)=\emptyset;$
\item[$(C4)$] $\{d,F_L(u)\}$ is LD for some $u\in F_H^{-1}(-d).$
\end{itemize}
\end{thm}

\begin{rem} $($The nonconvexity  of the joint-range set: a complete description$)$
If $F(\R^{n})$ is not convex then Theorems \ref{todo:0} and \ref{teor:1000}, and
its proof, imply the existence of
$d=(d_{1},\, d_{2})\neq0$, a change of variable $x=Cy-\overline{x}$ and
$k\in\R^{2}$ such that for all $x\in\R^{n}$, one has $F(x)=
\widetilde{F}(y)-k$ with $\widetilde{F}_{H}(y)=
\left(\underset{i=1}{\overset{m}{\sum}}y_{i}^{2}-y_{m+1}^{2}\right)d$
and  $\widetilde{F}_{L}\left(y\right)=\left(t_{1}y_{1}+t_{2}y_{m+1}\right)d_{\bot}$,
where $m$ may be possibly zero; moreover, there it holds
\begin{equation}\label{eq:1000}
y_{m+1}^{2}=1+\underset{i=1}{\overset{m}{\sum}}y_{i}^{2}\implies
t_{1}y_{1}+t_{2}y_{m+1}\neq0.
\end{equation}

In particular from \eqref{eq:1000} it follows $($using $y_{m+1}=1$ and
$y_{i}=0,\, i\neq m+1)$ that $t_{2}\neq0$. Furthermore, if
$t_{1}^{2}>t_{2}^{2}$, setting $t_{3}\doteq\sqrt{t_{1}^{2}-t_{2}^{2}}>0$,
the vector $y$ whose components are $y_{1}=\cfrac{t_{2}}{t_{3}}$,
$y_{m+1}=-\cfrac{t_{1}}{t_{3}}$ and  $y_{i}=0,\, i\neq1,\, m+1$, yields a
contradiction with \eqref{eq:1000}; proving that
$t_{1}^{2}\leq t_{2}^{2}\neq0$.
Thus, two possibilities arise:\\
$\bullet$ $t_{1}^{2}=t_{2}^{2}$, in which case, two sets come out as shown in
Figures \ref{nl:M2} and \ref{nl:M3}, up to translations and/or rotations.
Consider $l\doteq{\rm dim}({\rm ker}~A\cap {\rm ker}~B)$.

\begin{figure}[h]
     \begin{multicols}{2}
\begin{pspicture}(-3,-1)(1,4.5)
 \psset{yunit=1.5cm,xunit=1.5cm}
    \psaxes[ticks=none,labels=none,linestyle=dashed,linecolor=black,linewidth=1pt]{->}(0,0)(-2,-2)(2,2)
\pspolygon[linestyle=none,fillstyle=solid,fillcolor=blue!22!]
 (-1.5,-1.5)(1.5,-1.5)(1.5,1.5)(-1.5,1.5)(-1.5,-1.5)
 \psline[linecolor=white,linewidth=0.5pt](-1.5,0)(1.5,0)
\psdots[linecolor=blue,dotstyle=*,dotscale=1.2](0,0)
\psline[linecolor=black,linewidth=0.5pt]{->}(0.48,0.48)(1.2,0.48)
\psline[linecolor=black,linewidth=0.5pt]{->}(-0.48,0.48)(-1.2,0.48)
   \rput{0}(2.1,0){\color{red}{$u$}}
    \rput{0}(0,2.1){\color{red}{$v$}}
     \rput{0}(0.8,0.7){$d$}
     \rput{0}(-0.8,0.7){$d$}
    \rput{0}(0.15,-0.15){$k$}
    \end{pspicture}
    \vspace{2cm}
 \caption{$n-l\geq 2$}
 \label{nl:M2}
     \begin{pspicture}(-4,-1)(1,4.5)
 \psset{yunit=1.5cm,xunit=1.5cm}
    \psaxes[ticks=none,labels=none,linestyle=dashed,linecolor=black,linewidth=1pt]{->}(0,0)(-2,-2)(2,2)
\pspolygon[linestyle=none,fillstyle=solid,fillcolor=blue!22!]
 (-1.5,-1.5)(1.5,-1.5)(1.5,1.5)(-1.5,1.5)(-1.5,-1.5)
 \psline[linecolor=white,linewidth=0.5pt](-1.5,0)(0,0)
\psdots[linecolor=blue,dotstyle=*,dotscale=1.2](0,0)
\psline[linewidth=0.5pt,linecolor=blue](0,0)(1.5,0)
\psline[linecolor=black,linewidth=0.5pt]{->}(0.48,0.48)(1.2,0.48)
   \rput{0}(2.1,0){\color{red}{$u$}}
    \rput{0}(0,2.1){\color{red}{$v$}}
     \rput{0}(0.8,0.7){$d$}
    \rput{0}(0.15,-0.15){$k$}
    \end{pspicture}
    \vspace{2cm}
 \caption{$n-l\geq 3$}
 \label{nl:M3}
\end{multicols}
\end{figure}

\begin{figure}[Ht]
 \begin{multicols}{2}
 \begin{pspicture}(-3,-1)(1,4.5)
     \psset{yunit=1.5cm,xunit=1.5cm}
    \psaxes[linestyle=dashed,labels=none,ticks=none,linecolor=black,linewidth=0.5pt]{<->}(0,0)(-2,-2)(2,2)
    \psplot[linecolor=blue,linewidth=0.5pt]{-1.5}{0}{x neg sqrt }
  \psplot[linecolor=blue,linewidth=0.5pt]{-1.5}{0}{x neg sqrt neg}
  \psline[linecolor=black,linewidth=0.5pt]{->}(0.48,0.48)(1.2,0.48)
   \rput{0}(2.1,0){\color{red}{$u$}}
    \rput{0}(0,2.1){\color{red}{$v$}}
   \rput{0}(0.8,0.7){$d$}
    \rput{0}(0.15,-0.15){$k$}
    \end{pspicture}
   \vspace{2cm}
 \caption{$n-l=1$}
\label{nl=1}
  \begin{pspicture}(-4,-1)(1,4.5)
 \psset{yunit=1.5cm,xunit=1.5cm}
\psaxes[linestyle=dashed,labels=none,ticks=none,linecolor=black,linewidth=0.5pt]{<->}(0,0)(-2,-2)(2,2)
   \pscustom[linestyle=none,fillstyle=solid,fillcolor=blue!22!]{
\psplot{-1.5}{0}{x neg sqrt}
\psplot[linestyle=none]{0}{-1.5}{1.5}}
   \pscustom[linestyle=none,fillstyle=solid,fillcolor=blue!22!]{
\psplot{-1.5}{0}{x neg sqrt neg}
\psplot{0}{-1.5}{-1.5}}
\pspolygon[linestyle=none,fillstyle=solid,fillcolor=blue!22!](-0.001,1.5)(1.5,1.5)
(1.5,-1.5)(-0.001,-1.5)(0.001,1.5)
    \psplot[linecolor=blue,linewidth=0.5pt]{-1.5}{0}{x neg sqrt }
  \psplot[linecolor=blue,linewidth=0.5pt]{-1.5}{0}{x neg sqrt neg}
  \psline[linecolor=black,linewidth=0.5pt]{->}(0.48,0.48)(1.2,0.48)
   \rput{0}(0.8,0.7){$d$}
    \rput{0}(0.15,-0.15){$k$}
   \rput{0}(2.1,0){\color{red}{$u$}}
    \rput{0}(0,2.1){\color{red}{$v$}}
    \end{pspicture}
   \vspace{2cm}
 \caption{$n-l\geq 2$}
\label{nlMM2}
\end{multicols}
\end{figure}

\vskip 0.1truecm
\par\noindent
$\bullet$ $t_{1}^{2}<t_{2}^{2}$, in which case, we may assume $t_{1}=0$ up
to the change of variable $y_{1}'=\cfrac{t_{2}}{t_{3}}y_{1}+
\cfrac{t_{1}}{t_{3}}y_{m+1}$, $y_{m+1}'=\cfrac{t_{1}}{t_{3}}y_{1}+
\cfrac{t_{2}}{t_{3}}y_{m+1}$, $t_{3}=\sqrt{t_{2}^{2}-t_{1}^{2}}$; thus the set
may be have two possible forms as well, see Figures \ref{nl=1} and \ref{nlMM2},
up to translations and/or rotations.
\end{rem}

From the previous description, we immediately obtain $(a)$ of the next theorem.

\begin{thm}\label{dines:00}
Let $f, g$ be any quadratic functions as above. Then,
\begin{itemize}
\item[$(a)$] $F(\R^n)+\R_{++}d$ is convex for all non-null directions $d$ except
possibly at most two.
\item[$(b)$] $F(\R^n)+P$ is convex  for all convex cone with nonempty interior
$P\subseteq\R^2$. Consequently $F(\R^n)+{\rm int}~P$ is also convex.
\end{itemize}
\end{thm}
\begin{proof} $(a)$: It is a consequence of the following equalities:
$$F(\R^n)+\R_{++}d=F(\R^n)+(\R_{+}d+\R_{++}d)=(F(\R^n)+\R_{+}d)+\R_{++}d.$$
$(b)$: Since ${\rm int}~P\not=\emptyset$, we can choose $d\in P$ such that
$F(\R^n)+\R_+d$ is convex. The result follows by noting that
$$F(\R^n)+P=F(\R^n)+(\R_{+}d+P)=(F(\R^n)+\R_{+}d)+P.$$
Thus, ${\rm int}(F(\R^n)+P)=F(\R^n)+{\rm int}~P$ is also convex.
\end{proof}

\begin{thm}\label{di:alter}
Let $d\in\R^2$, $d\not=0$. Then either $d=(d_1,d_2)\not\in-{\rm bd}~F_H(\R^n)$
or $d_2A-d_1B$ is semidefinite $($positive or negative$)$.
\end{thm}
\begin{proof}
Assume that $d_2A-d_1B$ is not semidefinite, that is, there exist
$x_1, x_2\in\R^n$ such that $\langle d_\perp, F_H(x_1)\rangle<0<
\langle d_\perp, F_H(x_2)\rangle$. Then, it is not difficult to check that
either $-d\in{\rm int}~F_H(\R^n)$ or $-d\not\in \overline{F_H(\R^n)}$, which
mean $-d\not\in{\rm bd}~F_H(\R^n)$.
\end{proof}


\section{Nonconvex quadratic programming with one single inequality or equality
constraint} \label{secc5}
In this section we are concerned with the following quadratic minimization 
problem:
\begin{equation}\label{probpq_min0}
\mu\doteq {\rm inf}\{ f(x):~~g(x)\in -P,~x \in \R^n\},
\end{equation}
where $P$ is either $\R_+$ or $\{0\}$, and $f, g:\R^n\to \R$ are any quadratic 
functions given by
\begin{equation}\label{qua:0}
f(x)\doteq\dfrac{1}{2} \langle Ax,x\rangle+\langle a,x\rangle+k_1;~~
g(x)\doteq\dfrac{1}{2} \langle Bx,x\rangle+\langle b, x\rangle+k_2,
\end{equation}
with $A, B\in{\mathcal S}^n$, $a,~b\in \R^n$ and $k_1, k_2\in\R$.\\
The (Lagrangian) dual problem associated to \eqref{probpq_min0} is defined by
\begin{equation}\label{probdq_min0}
\nu\doteq\sup_{\lambda\in P^*}\underset{x\in\R^n}{\rm inf}
\{ f(x)+\lambda g(x)\}.
\end{equation}
Clearly we obtain
\begin{equation}\label{g:d}
\underset{x\in C}{\rm inf}\{ f(x)+\lambda g(x)\}\leq \mu,~~\forall~\lambda\in P^*.
\end{equation}
It is said that \eqref{probpq_min0} has the strong duality property, or simply
that strong duality holds for \eqref{probpq_min0}, if $\mu=\nu$ and
problem \eqref{probdq_min0} admits any solution.\\
Thus, in case $\mu=-\infty$, there is no duality gap since $\nu=-\infty$ as well, 
and from \eqref{g:d}, we conclude that any element in $P^*$ is a solution for the 
problem \eqref{probdq_min0}. Hence, strong duality always holds for 
\eqref{probpq_min0} provided $\mu=-\infty$.

Setting  $F\doteq (f,g)$, $\mu\in\R$ means 
\begin{equation}\label{id:01}
[F(\R^n)-\mu(1,0)]\cap-(\R_{++}\times P)=\emptyset,
\end{equation}
or equivalently,
\begin{equation}\label{id:03}
[F(\R^n)+(\R_+\times P)-\mu(1,0)]\cap-(\R_{++}\times\{0\})=\emptyset.
\end{equation}
Hence, in case of one inequality constraint, i. e., $P=\R_+$, \eqref{id:03} becomes
$$[F(\R^n)+\R_+^2-\mu(1,0)]\cap-(\R_{++}\times\{0\})=\emptyset;$$
whereas in case $P=\{0\}$, that is, under one single equality constraint, 
\eqref{id:03} reduces to   
$$[F(\R^n)+\R_+(1,0)-\mu(1,0)]\cap-(\R_{++}\times\{0\})=\emptyset.$$

Thus, we are interested only in the convexity of $F(\R^n)+\R_+(1,0)$ since 
$F(\R^n)+\R^2$ is always convex by Theorem \ref{dines:00}.

By particularizing $d=(1,0)$ in Theorem  \ref{teor:1000}, it yields the following
corollary.

\begin{cor} \label{coro:scon0}
Let $f, g$ be quadratic functions as in \eqref{qua:0}.
Then,
\begin{itemize}
\item[$(a)$] $F(\R^n)+\R_+(1,0)$ is  nonconvex if and only if
$$B=0,~\{a,b\}\subseteq({\rm ker}~A)^\perp=A(\R^n),~
\{u\in\R^n:~\langle Au,u\rangle<0\}\not=\emptyset,~~{\it and}$$
\begin{equation}\label{bubu}
\langle Au,u\rangle<0\Longrightarrow \langle b,u\rangle\not=0;
\end{equation}
\item[$(b)$] if $F(\R^n)+\R_+(1,0)$ is  nonconvex, one obtains
\begin{enumerate}
\item[$(b1)$] $\not\exists~(\lambda,\rho)\in\R^2$, $f(x)+\lambda g(x)\geq\rho $,
$\forall~x\in\R^n$ and therefore
$$\displaystyle{\inf_{x\in\R^n}}[f(x)+\lambda g(x)]=-\infty,~\forall~\lambda\in\R;$$
\item[$(b2)$] $\exists~x_i\in\R^n$, $i=1, 2$, $g(x_1)<0<g(x_2).$
\end{enumerate}
\end{itemize}
\end{cor}
\begin{proof} $(a)$ is a consequence from Theorem \ref{teor:1000} with $d=(1,0)$.\\
Assume now that $(b1)$ does not hold, then $A\succcurlyeq 0$, contradicting $(a)$.
Then, the second part immediately follows.\\
$(b2)$ It follows from \eqref{bubu}.
\end{proof}

\begin{rem} \label{obs00} As a counterpart to the preceding corollary, we deduce 
that $F(\R^n)+\R_+(1,0)$ is convex if, and  only if any of the following conditions 
is satisfied:
\begin{itemize}
\item[$(C1)$] $B=0$ and $\exists~u\in{\rm ker}~A$:
$(\langle a,u\rangle,\langle b,u\rangle)\not=(0,0);$
\item[$(C2)$] $B\not=0;$
\item[$(C3)$] $B=0$ and $A\succcurlyeq 0;$
\item[$(C4)$] $B=0$ and $\exists~u\in\R^n$: $\langle Au,u\rangle<0$ and
$\langle b,u\rangle=0.$
\end{itemize}
The latter condition implies $\mu=-\infty$ $($with $P=\{0\})$ by Proposition 
\ref{an:b0}.
\end{rem}

\subsection{The nonstrict version of S-lemma (Finsler's theorem), a 
strong duality and optimality conditions revisited}\label{s:lema}

The validity of S-lemma with equality ($P=\{0\}$) is  characterized in Theorems 1 
and 3 in \cite{xws2015} by a completely different approach. Our purpose  is to 
provide some sufficient conditions for that validity as a consequence of our 
results from Section \ref{secc4}. These conditions will be  expressed in a 
different way than that in \cite{xws2015}. 

The case $P=\R_+$ already appears in \cite{yaku1971, yaku1977} known as the
$S$-procedure, see also \cite[Corollary 5]{sz2003},
\cite[Proposition 3.1]{lsz2004}, \cite[Theorem ~2.2]{polter2007},
\cite[Corollary 3.7]{jleeli2009}, and a slight variant in
\cite[Theorem 3.4]{ffb-carcamo2013}. Some extensions of the
$S$-procedure in a different direction may be found in \cite{derin2006}.

We now establish that sufficient conditions for the validity of S-lemma
for inhomogeneous quadratic functions. Set 
$$K_P\doteq\{x\in\R^n:~g(x)\in -P\}.$$
\begin{thm} $(S$-lemma$)$ \label{slema:p}
Let $P$ be either $\R_+$ or $\{0\}$, $K_P\not=\emptyset$ and
$f, g:\R^n\to \R$ be any quadratic functions as in \eqref{qua:0}, satisfying
$0\in{\rm ri}(g(\R^n)+P)$. In case $P=\{0\}$, assume additionally that
$g\not\equiv0$ and that
any of the conditions $(Ci)$, $i=1, 2, 3$, holds.
Then, $(a)$ and $(b)$ are  equivalent:
\begin{itemize}
\item[$(a)$] $x\in \R^n,~g(x)\in-P\Longrightarrow f(x)\geq 0$.
\item[$(b)$] There is $\lambda\in P^*$ such that
$f(x)+\lambda g(x)\geq 0,~~\forall~x\in \R^n.$
\end{itemize}
\end{thm}
\begin{proof} Obviously $(b)\Longrightarrow(a)$ always holds.
Assume therefore that $(a)$ is satisfied. This means that
$0\leq \mu\doteq\underset{g(x)\in-P}{\inf}~f(x)$. It follows that 
\eqref{id:03} holds.
By our previous discussion $F(\R^n)+(\R_{+}\times P)$ is convex, and
so by a separation theorem, there exist
$(\gamma,\lambda)\in\R^2\setminus\{(0,0)\}$ and $\alpha\in\R$ such that
\begin{equation*}\label{d00}
\gamma(f(x)-\mu+p)+\lambda (g(x)+q)\geq \alpha\geq \gamma u,~
\forall~x\in\R^n,~\forall~p\in\R_+,~\forall~q\in P,~\forall~u<0.
\end{equation*}
This yields $\alpha\geq 0$,  $\gamma\geq 0$ and $\lambda\in P^*$, which imply
$\gamma(f(x)-\mu)+\lambda g(x)\geq 0,~~\forall~x\in\R^n,$
that is, $\gamma f(x)+\lambda g(x)\geq \gamma\mu\geq 0,~~\forall~x\in \R^n.$
The Slater-type condition gives $\gamma>0$, completing the proof of the
theorem.
\end{proof}



\begin{rem} \label{obs:slema}
$($Comparison with the S-lemma with equality given in
$\cite{sz2003, lsz2004, polter2007})$\\
Here, our discussion refers to $P=\{0\}$. The $S$-lemma in
$\cite[Corollary~ 6]{sz2003}$, see also $\cite[Proposition~3.2]{lsz2004}$, or
$\cite[Proposition~ 3.1]{polter2007}$, asserts that $(a)$
and $(b)$ are equivalent under the assumptions $(i)$ and $(ii):$\\
$(i)$ $g$ is strictly convex $($or strictly concave$)$ and \\
$(ii)$ there exist $x_i\in\R^n,~i=1, 2$ such that $g(x_1)<0<g(x_2)$.\\
We first observe that such a result cannot be applied to homogeneous quadratic
functions $($which only requires $(ii)$, see $\cite[Theorem~ 2.3]{more}$ or
$\cite{ham1999})$, as one can notice it directly. On the contrary, our $S$-lemma
recovers that result, since $(i)$ implies $(C3)$: indeed
$\langle Bu,u\rangle=0$ implies $u=0$, and so $F_H^{-1}(-1,0)=\emptyset$.
Secondly, it is easy to check that $(ii)$ is equivalent to:\\
$(ii')$ $0\in{\rm ri}~g(\R^n)$ and $g\not\equiv0$.\\
On the other hand, our Theorem \ref{slema:p} applies to Example
\ref{ej:s:lema} but Proposition 3.1 in $\cite{polter2007}$ does not, since
$g$ in this case is neither strictly convex nor strictly concave.
\end{rem}


A characterization of the validity of S-lemma, for fixed $g$ with $P=\R_+$, for
each quadratic function $f$, is given in \cite[Theorem 3.1]{jhl_optimeng2009}.

An immediate new result on strong duality, when $P=\{0\}$, arises from the
previous theorem.
\begin{thm}\label{moreg}
Let $P$ be either $\R_+$ or $\{0\}$; $f, g:\R^n\to \R$ be any quadratic functions, 
as above, satisfying $0\in{\rm ri}(g(\R^n)+P)$ with $\mu\in\R$. In case $P=\{0\}$, 
assume additionally that $g\not\equiv0$ and that any of the conditions $(Ci)$, 
$i=1, 2, 3$, holds.
Then, strong duality holds for the problem
\eqref{probpq_min0}, that is, there exists $\lambda^*\in P^*$ such that
\begin{equation}\label{sd0:00}
\inf_{g(x)\in-P}f(x)=\inf_{x\in\R^n}[f(x)+\lambda^* g(x)].
\end{equation}
\end{thm}
\begin{proof} From $\mu\in\R$, we infer that there is no $x\in \R^n$ such that
$f(x)-\mu<0,~g(x)\in-P$. Then, we apply Theorem \ref{slema:p} to conclude
with the proof.
\end{proof}

We single out the case $P=\{0\}$ to obtain a new characterization of the validity
of strong duality for inhomogenoeus quadratic functions under Slater-type condition.
Its proof follows from the previous theorem and Corollary \ref{coro:scon0}.

\begin{cor}\label{moregg}
Let $P=\{0\}$; $f, g:\R^n\to \R$ be as above satisfying
$g(x_1)<0<g(x_2)$ for some $x_1, x_2\in\R^n$.
Then, 
$$\mu\in\R~{\it and~ strong~ duality~ holds~ for}~\eqref{probpq_min0}~
\Longleftrightarrow~
\nu\in\R~{\it and}~F(\R^n)+\R_+(1,0)~{\it is~ convex}.$$
\end{cor}

For the convexity of $F(\R^n)+\R_+(1,0)$, we refer to Remark \ref{obs00}.

In case we have strong duality with $\mu=-\infty$ it is possible that
$F(\R^n)+\R_+(1,0)$ may be nonconvex. The following example shows this fact.
\begin{ex} Take $f(x_1,,x_2)=x_1x_2$, $g(x_1,x_2)=x_1+1$. Then
$\mu=-\infty$. Moreover, since $(-1,2)=F(1,-1)$, $(-1,0)=F(-1,1)$ but 
$(-1,1)\not\in F(\R^n)+\R_+(1,0)$, the latter set is nonconvex.
\end{ex}

In connection to the previous result, we must point out that when 
$F(\R^n)+\R_+(1,0)$ is not convex, then $g(x_1)<0<g(x_2)$ for some 
$x_1, x_2\in\R^n$ and $\nu=-\infty$
by Corollary \ref{coro:scon0}.

Next example shows that a Slater-type condition is necessary.
\begin{ex} \label{ej:sinsd1}
Let us consider $f(x_1,x_2)=x_1+x_2$ and
$g(x_1,x_2)=(x_1+x_2)^2$. One  can deduce that
there is no duality gap. It is easy to get
$$F_H(\R^2)=\R_+(0,1),~~F(\R^2)=\{(u,v)\in\R^2:~v=u^2\},~~g(\R^2)=\R_+.$$
Thus $F_H^{-1}(-1,0)=\emptyset$ $($implying that
$F(\R^2)+\R_+(1,0)$ is convex$)$, but $0\not\in {\rm ri}(g(\R^2)+P)$.
Moreover, the strong duality does not hold, since for any $\lambda>0$, the
inequality
$$x_1+x_2+\lambda(x_1+x_2)^2\geq 0,~~\forall~(x_1,x_2)\in\R^2,$$
is impossible.
\end{ex}

Strong duality results (with $P=\R_+$) were also derived in
\cite[Theorem 3.2]{jhl_optimeng2009} and \cite[Theorem 3.2]{jeya-li2011}, with a
different perspective: in both papers it is  characterized the validity of such
a result for each quadratic function $f$.



By applying the previous corollary, we obtain a necessary and sufficient
optimality condition, which is an extension of
Theorem 3.2 in \cite{more}, where the assumption $B\not=0$ (which is our
condition $(C3)$) is imposed when $P=\{0\}$.
The case $P=\R_+$ was already considered in
\cite[Proposition 3.3]{jeya-rubi2007},
\cite[Theorem 3.4]{more}, \cite[Theorem 3.8]{jleeli2009},
\cite[Theorem 3.15]{ffb-carcamo2013}.

\begin{cor} \label{more0}
Let $P$ be either $\R_+$ or $\{0\}$, $K_P\not=\emptyset$ and
$f, g:\R^n\to \R$ be any quadratic functions, as above, satisfying
$0\in{\rm ri}(g(\R^n)+P)$. In case $P=\{0\}$, assume additionally that
$g\not\equiv0$ and that any of the conditions $(Ci)$, $i=1, 2, 3$, holds.
Then, the following assertions are equivalent:
\begin{itemize}
\item[$(a)$] $\bar x\in\underset{g(x)\in - P}{\rm argmin}~f;$
\item[$(b)$] $\exists~\lambda^*\in P^*$ such that
$\nabla f(\bar x)+\lambda^*\nabla g(\bar x)=0$ and
$A+\lambda^* B\succcurlyeq 0$.
\end{itemize}
\end{cor}
\begin{proof} It follows a standard reasoning by applying the previous
corollary.
\end{proof}

The last corollary deserves to make some remarks.

\begin{rem} \label{obs:more0} We consider $P=\{0\}$.\\
$(i)$ Next example, taken from  $\cite{more}$, shows that our set of assumptions
$(Ci)$, $i=1, 2, 3$ is, in some sense, optimal. Consider
$$\min\{x_1^2-x_2^2:~x_2=0\}.$$
Then, $F_H(\R^2)= \R(1,0)$. Observe that $B=0$,
${\rm ker}~A=\{ (0,0)\}$, $F_L({\rm ker}~A)=\{(0,0)\}$
$F_H^{-1}(-1,0)\not=\emptyset$ and $\{(1,0), F_L(u)\}$ is LI for all
$u\in F_H^{-1}(-1,0)$. Hence $(C1)$, $(C2)$, $(C3)$ and $(C4)$ do not hold, in other
words, $F(\R^n)+\R_+(1,0)$ is nonconvex.
We easily see  that the KKT conditions is not satisfied for the optimal solution
$\bar x=(0,0)$.\\
$(ii)$  Our Corollary \ref{more0} applies to situations that are not covered by
Theorem 3.2 in $\cite{more}$.
In fact, let us consider $\min\{x_1^2:~x_2=0\}$. Then
$F_H(\R^2)=\R_+(1,0)$, which gives $F_H^{-1}(-1,0)=\emptyset$.
Thus our previous corollary is applicable, but not that in $\cite{more}$ since
$B=0$.
\end{rem}

For completeness we establish a characterization of solutions when
$P=\{0\}$ and the Slater-type condition: $0\in{\rm ri}~g(\R^n)$ and
$g\not\equiv 0$ (which is equivalent to $(ii)$ in Remark \ref{obs:slema}) fails.
We only consider $g(x)\geq 0$ for all $x\in\R^n$, the case $g(x)\leq 0$ for
all $x\in\R^n$ is similar. This implies that
\begin{equation}\label{rest:g0}
K_P=K_0=\{x\in\R^n:~g(x)=0\}=\underset{\R^n}{\rm argmin}~g,
\end{equation}
provided $K_P\not=\emptyset$. It is known that
\begin{equation}\label{eqv0}
\bar x\in \underset{\R^n}{\rm argmin}~g\Longleftrightarrow
[B\succcurlyeq 0~{\rm and}~B\bar x+b=0].
\end{equation}
This leads to the following corollary.

\begin{cor} \label{more:sins}
$($Slater condition fails$)$ Let $f$, $g$ be any quadratic functions
and $\bar x\in K_P$ with $P=\{0\}$. Assume that $g(x)\geq 0$ for all $x\in \R^n$.
Then $F(\R^n)+\R_+(1,0)$ is convex, and the following statements are equivalent:
\begin{itemize}
\item[$(a)$] $\bar x\in\underset{K_P}{\rm argmin}~f;$
\item[$(b)$] $B\succcurlyeq 0$, $A$ is positive semidefinite on
${\rm ker}~B$, and $\exists~v\in\R^n$ such that
$$A\bar x+a+Bv=0,~B\bar x+b=0.$$
\end{itemize}
\end{cor}

\subsection{The ND property and the minimization problem}

Next result describes some necessary conditions for having the optimal value of
problem \eqref{probpq_min0} to be finite.

\begin{prop} \label{an:b0}
Assume that $\mu$ is finite. The following assertions hold.
\begin{itemize}
\item[$(a)$] if $P=\R_+$ then
\begin{equation}\label{cond:gene}
v\not=0,~\langle Bv,v\rangle\leq 0~\Longrightarrow~
\langle Av,v\rangle\geq 0.
\end{equation}
\item[$(b)$]  if $P=\{0\}$ and there exists $v\in\R^n$ satisfying
$\langle Bv,v\rangle= 0$ and $\langle Av,v\rangle<0$, then $Bv=0$
and either
$$\langle b,v\rangle>0~{\rm and}~
f(x+tv)\to -\infty~~{\rm as}~|t|\to+\infty,~
~\forall~x\in\R^n,
$$
or

$$\langle b,v\rangle<0~{\rm and}~
f(x+tv)\to -\infty~~{\rm as}~|t|\to+\infty,
~~\forall~x\in\R^n.
$$
\end{itemize}
\end{prop}
\begin{proof} $(a)$ It is Proposition 3.6 in \cite{ffb-carcamo2013}.\\
$(b)$: We obtain, given any $x\in \R^n$,  $f(x+tv)\to-\infty$ for all
$|t|\to+\infty$. Then there exists $t_1>0$ such that
$g(x+tv)=g(x)+t\langle\nabla g(x), v\rangle\not= 0$ for all $|t|>t_1$. By
splitting both expressions, we obtain either
$$f(x+tv)\to -\infty~~{\rm as}~t\to+\infty,~
f(x-tv)\to -\infty~~{\rm as}~t\to+\infty,~{\rm and}~\langle\nabla g(x),v\rangle>0,
$$
or
$$f(x+tv)\to -\infty~~{\rm as}~t\to+\infty,
f(x-tv)\to -\infty~~{\rm as}~t\to+\infty,~{\rm and}~\langle\nabla g(x),v\rangle<0,
$$
from which the desired results follow.
\end{proof}

\begin{thm}\label{exis:q} Consider problem \eqref{probpq_min0} and let
$\mu\in\R$. Assume that
\begin{equation}\label{lscq:00}
F_H(v)=0\Longrightarrow v=0.
\end{equation}
\begin{itemize}
\item[$(a)$] If $P=\{0\}$ then every minimizing sequence is bounded, and so
 $\underset{g(x)=0}{\rm argmin}~f$ is nonempty and compact.
\item[$(b)$] If $P=\R_+$ then $\underset{g(x)\leq 0}{\rm argmin}~f$ is
nonempty. More precisely, every unbounded minimizing sequence $x_k\in K_P$
satisfying $\|x_k\|\to +\infty$, $\dfrac{x_k}{\|x_k\|}\to v$, yields the
existence of $\bar x\in\underset{\R^n}{\rm argmin}~f$ such that, for some
$t_0>0$,
\begin{equation}\label{pseudo:q}
\bar x+tv\in\underset{g(x)\leq 0}{\rm argmin}~f,~~\forall~|t|>t_0.
\end{equation}
Furthermore, $Av=0$ and $\langle a,v\rangle=0$.
\end{itemize}
\end{thm}
\begin{proof} $(a)$: Case $P=\{0\}$: take any minimizing sequence $x_k\in K_P$.
Suppose that $\sup_k\|x_k\|=+\infty$. Up to a subsequence, we may assume that
$\|x_k\|\to +\infty$ and $\dfrac{x_k}{\|x_k\|}\to v$. From $g(x_k)=0$ and
$f(x_k)\to \mu$ it follows that
$\langle Bv,v\rangle=0$ and $\langle Av,v\rangle=0$. By assumption, $v=0$,
reaching a contradiction. Hence every minimizing sequence is bounded.\\
$(b)$: Case $P=\R_+$: take any minimizing sequence $x_k\in K_P$. If
$\sup_k\|x_k\|<+\infty$, we get that every limit point of $\{x_k\}$ yields a
solution to \eqref{probpq_min0}, as usual. \\
Take now any minimizing sequence $x_k$ such that $\|x_k\|\to +\infty$ and
$\dfrac{x_k}{\|x_k\|}\to v$. From $g(x_k)\leq 0$
and $f(x_k)\to \mu$ it follows that  $\langle Bv,v\rangle\leq 0$ and
$\langle Av,v\rangle=0$. By assumption,
$\langle Bv,v\rangle< 0$ and $\langle Av,v\rangle=0$. \\
Thus, by writting, for any $x\in\R^n$,
$g(x+tv)=g(x)+t\langle\nabla g(x),v\rangle+
\frac{1}{2}t^2\langle Bv,v\rangle $, we conclude
that $g(x+tv)<0$ for all $|t|>t_1$, for some $t_1$ depending of $x$, and
therefore $f(x+tv)\geq\mu$ for all $|t|\geq t_1$. Since
$\mu\leq f(x+tv)=f(x)+t\langle\nabla f(x),v\rangle$, we deduce that
$\langle\nabla f(x),v\rangle=0$, and so $\mu\leq f(x+tv)=f(x)$ for all $t\in\R$.
The former implies $Av=0$ and $\langle a,v\rangle=0$, and the latter gives that
$\displaystyle\mu=\inf_{x\in\R^n} f(x)$. Hence
$A\succcurlyeq 0$ and there exists $\bar x\in\underset{\R^n}{\rm argmin}~f$
such that $A\bar x+a=0$. Moreover, since $f(\bar x+tv)=f(\bar x)=\mu$ for all
$t\in\R$, we infer that $g(\bar x+tv)<0$ for all $|t|>t_0$, and so
\eqref{pseudo:q} is satisfied.
\end{proof}

\begin{rem} Part $(b)$ of the previous theorem provides explicit solutions to
\eqref{probpq_min0}. Indeed, it is well known that
$\bar x\in\underset{\R^n}{\rm argmin}~f$ if, and only if
$A\succcurlyeq 0$ and $A\bar x+a=0$. By using the pseudoinverse of More-Penrose
of any matrix, one obtains that $x_0 =-A^\dagger a$, with $A^\dagger$ being
such a pseudoinverse of $A$, is the unique solution with minimal norm.  Thus,
by taking $t$ sufficiently large, $x_0+tv$ is a solution for the
problem \eqref{probpq_min0}.
\end{rem}
The next instance shows that without assumption \eqref{lscq:00} the set of
minima may be empty.

\begin{ex} \label{ej:s:lema}
Let $P$ be either $\{0\}$ or $\R_+$ and take
$$A=\begin{pmatrix}
2 & 0\cr
0 & 0
\end{pmatrix},~~
B=\begin{pmatrix}
0 & 2\cr
2 & 0
\end{pmatrix},~~
a=\begin{pmatrix}
0\cr
0
\end{pmatrix},~~
b=\begin{pmatrix}
0\cr
0
\end{pmatrix},~~k_1=0,~~k_2=1.
$$
Then $F(\R^2)=(0,1)+F_H(\R^2),~~
F_H(\R^2)=\{(0,0)\}\cup (\R_{++}\times\R).$ In addition,
one can check that $0=\mu\doteq\min\{x_1^2:~2x_1x_2+1\in-P\}$,
\eqref{lscq:00} is not satisfied and $\underset{g(x)\in-P}{\rm argmin}~f=\emptyset$.
\end{ex}

\section{Some historical notes for a pair of quadratic forms}\label{secc6}


We will concern only with a pair of quadratic forms in $\R^n$, and use the
notation introduced in Section \ref{secc2}. It seems to be the
convexity Dines theorem was conceived once Dines awares the following result,
known as (strict) Finsler's theorem: if
\begin{equation}\label{ecua1:f}
[0\not=v,~\langle Bv,v\rangle=0]~\Longrightarrow \langle Av,v\rangle>0
\end{equation}
then
\begin{equation}\label{ecua:f111}
\exists~\lambda\in\R,~A+\lambda B\succ 0,
\end{equation}
and believed that convexity must be present. The previous result was proved
first, as far as  we know, by Finsler in \cite{fins1937}, and re-proved in
\cite{albert1938, reid1938, heste1968, ham1999} (a extension to more than two
matrices appears in \cite{heste-mc1940}). That result is a kind
of S-lemma which originally read as follows: assuming that
$\langle B\bar v,\bar v\rangle<0$ for some $\bar v$, then
\begin{equation}\label{ecua2:f}
\langle Bv,v\rangle\leq 0~\Longrightarrow \langle Av,v\rangle\geq 0
\end{equation}
is equivalent to
\begin{equation}\label{ecua1:f111}
\exists~\lambda\geq 0,~A+\lambda B\succcurlyeq 0.
\end{equation}
This lemma was proved by Yakuvobich \cite{yaku1971, yaku1977}. Since then,
several variants of it and possible connections with  well-known
properties of matrices have been appeared. A nice survey  about the S-lemma is
presented in \cite{polter2007}; whereas the mentioned properties treated in detail
may be found, for instance, in \cite{greub1967, hj1985}, see also
\cite{uhlig1979}.

In what follows we list some of the main properties useful in the study of
quadratic forms.

\begin{itemize}
\item[$(a)$] SD;
\item[$(b)$]  $\exists~t_1,~ t_2\in\R,~t_1A+t_2B\succ 0;$
\item[$(c)$]  $\exists~t\in\R,~A+t B\succ 0;$
\item[$(d)$] $[0\not=v,~\langle Bv,v\rangle= 0]~\Longrightarrow
\langle Av,v\rangle>0;$
\item[$(e)$] ND; 
\item[$(f)$] $ \langle Bv,v\rangle= 0~\Longrightarrow \langle Av,v\rangle
\geq 0;$
\item[$(g)$]  $\exists~t\in\R,~A+t B\succcurlyeq 0;$
\item[$(h)$]  $F_H(\R^n)=\R^2$.
\end{itemize}

The relationship between these properties are given below:
\begin{itemize}
\item[$\bullet$]  $(b)\Longrightarrow (a)$, see \cite[Theorem 7.6.4]{hj1985};
\item[$\bullet$] $(c)\Longleftrightarrow(d)$, see \cite{fins1937}, also
\cite{albert1938}, \cite[Corollary 2, page 498]{dines1941},
\cite{calabi1964, ham1999}; a different  proof may be found in
\cite[Theorem 2.2]{more};
\item[$\bullet$] ($n\geq 3$) $(e)\Longrightarrow (a)$ it is attributed to Milnor,
\cite[page 256]{greub1967}, see also \cite[Theorem 2.1]{uhlig1972};
\item[$\bullet$] $[F_H(\R^n)\not=\R^2$ and $(e)]\Longleftrightarrow (b)$,
see \cite[Corollary 1, page 498]{dines1941};
\item[$\bullet$] ($n\geq 3$) $(h)\Longrightarrow$ ND is not true, see
\cite[Corollary, page 401]{heste1968};
\item[$\bullet$] ($n\geq 3$) $(e)\Longleftrightarrow (b)$, see
\cite{fins1937}, also  \cite{calabi1964};
\item[$\bullet$] ($B$ indefinite) $(f)\Longleftrightarrow(g)$, see
\cite[Theorem 2.3]{more} and \cite{ham1999}.
\end{itemize}

Finally, in \cite{yan2010} some relationships between $(f)$, $(d)$  and $(e)$
and the Yakuvobich $S$-lemma (for a pair of quadratic forms), are estalished.
They are related to the non-strict Finsler's, strict Finsler's and
Finsler-Calabi's theorem, respectively.


\end{document}